\documentclass[review]{elsarticle}
\usepackage{amsfonts}
\usepackage{lineno,hyperref}
\modulolinenumbers[5]
\usepackage{amsmath, amssymb , graphicx }
\usepackage{algorithm} 
\usepackage[noend]{algpseudocode}

\newenvironment{proof}[1][Proof.]{ \begin{trivlist}
\item[\hskip \labelsep {\bfseries #1}]}{\end{trivlist}}
\usepackage{tikz}
\usetikzlibrary{shapes.geometric, arrows}
\usepackage{booktabs}
\usepackage{isomath}
\usepackage{latexsym,    amscd, amsfonts, enumerate, color}
\journal{Journal of Computational and Applied Mathematics}

\newtheorem{theorem}{Theorem }
\newtheorem{remark}{Remark  }
\newtheorem{example}{Example  }
\newtheorem{corollary}{Corollary  }
\newtheorem{definition}{Definition  }

\makeatletter
\newdimen\bibindent
\newdimen\betweenumberspace    
\newdimen\headlineindent            
\makeatother

\begin{document}

\begin{frontmatter}

\title{ Hybrid Radial Kernels for Solving Weakly Singular Fredholm Integral Equations: Balancing Accuracy and Stability in Meshless Methods
}
 \author[a] {Davoud Moazami}
\ead{davoud.moazami@stu.malayeru.ac.ir}

 \author[a] {Mohsen Esmaeilbeigi \corref{cor1}}
\ead{m.esmaeilbeigi@malayeru.ac.ir}

\author[a] {Tahereh Akbari}
\ead{tahereh.akbari@stu.malayeru.ac.ir}

 \cortext[cor1]{Corresponding author}
\address[a]{Faculty of Mathematical and Statistical Sciences, Malayer University, Malayer, Iran}


\begin{abstract}
Over the past few decades, kernel-based approximation methods had achieved astonishing success in solving different problems in the field of science and engineering. However, when employing the direct or standard method of performing computations using infinitely smooth kernels, a conflict arises between the accuracy that can be theoretically attained and the numerical stability. In other words, when the shape parameter tends to zero, the operational matrix for the standard bases with infinitely smooth kernels become severely ill-conditioned. This conflict can be managed applying hybrid kernels. The hybrid kernels extend the approximation space and provide high flexibility to strike the best possible balance between accuracy and stability. In the current study, an innovative approach using hybrid radial kernels (HRKs) is provided to solve   weakly singular Fredholm integral equations  (WSFIEs) of the second kind in a meshless scheme.  The approach employs hybrid kernels built on dispersed nodes as a basis within the discrete collocation technique.  This method transforms the problem being studied  into a linear system of algebraic equations.  Also, the particle swarm optimization (PSO) algorithm is utilized  to calculate the optimal  parameters for the hybrid kernels, which is based on minimizing the maximum absolute error (MAE).  We also study the error estimate of the suggested scheme. Lastly,  we assess the accuracy and validity of the hybrid technique by carrying out various numerical experiments. The numerical findings show that the estimates obtained from hybrid kernels are significantly more accurate in solving WSFIEs compared to pure kernels. Additionally, it was revealed that the hybrid bases remain stable across various values of the shape parameters.
\end{abstract}


\end{frontmatter}

\section{Introduction}
  The primary aim  of this research is  to introduce a robust and stable computational technique  for approximating the second kind of  WSFIEs
\begin{equation}\label{eq1}
u(\mathbfit{x})-\lambda \int_{\Omega} K(\mathbfit{x}, \mathbfit{t}) u(\mathbfit{t}) d\mathbfit{t}=f(\mathbfit{x}),  \quad \mathbfit{x}=(x_1,x_2, \ldots, x_d), \mathbfit{t}=(t_1, t_2, \ldots, t_d) \in \Omega \subset \mathbb{R}^d, \quad  d\mathbfit{t}=dt_1 \ldots dt_d,
\end{equation}
in which the  {\color{red}right-hand side  (rhs) } function $f$ and the kernel function
\begin{align*}
K(\mathbfit{x}, \mathbfit{t})=\mathcal{R}(\mathbfit{x}, \mathbfit{t}) \mathcal{S}(\mathbfit{x}, \mathbfit{t})
\end{align*}
are given, $u$  represents the unknown function,  {\color{red} $ \lambda \in \mathbb{R} \setminus \left\{ 0 \right\} $   } and $\Omega$ is a closed bounded domain in $\mathbb{R}^d$. 
We also assume that $\mathcal{S}$ is several times continuously differentiable on $\Omega \times \Omega$ and 
the known function $\mathcal{R}$ has an infinite singularity in  {\color{red}
$\left\{ \left( \mathbfit{x}, \mathbfit{t} \right) \in \Omega \times \Omega: \mathbfit{x}= \mathbfit{t} \right\},$ 
the most significant  instances being}
\[ \mathcal{R}(\mathbfit{x}, \mathbfit{t})= \ln \Bigg ( \sqrt{ \sum_{i=1}^{d}   (x_i-t_i)^2  }  \Bigg )  \]
and
\[  \mathcal{R}(\mathbfit{x}, \mathbfit{t})= \Big \{  \sum_{i=1}^{d} (x_i-t_i) ^2  \Big \}^{\frac{\varsigma}{2}},  \]
for some $-1 < \varsigma <0,$ and variants of them \cite{A1}. 

WSFIEs play a crucial role in a wide range of engineering disciplines, such as, the analysis of electrostatic and low-frequency electromagnetic challenges \cite{1}, techniques for calculating the conformal mapping of specific regions \cite{2}, analyzing the propagation of acoustic and elastic waves \cite{3,4}, formulating problems related to radiative equilibrium and heat transfer \cite{5}, describing hydrodynamic interactions among elements in a polymer chain \cite{6}, addressing the scattering of surface water waves by a vertical barrier featuring a gap \cite{7}, and studying Dirichlet's problem associated with logarithmic potentials \cite{80},  and so on. For instance, the problem of ascertaining the cross-sectional distribution of current $u(x,t)$ in an infinitely long, slender conducting bar that carries an alternating current  \cite{7x} can be  defined by the following two dimensional WSFIEs:
\[  u(x,t)- \frac{\tau h i \nu}{2 \pi} \int_{\Omega} \Bigl[  \ln \sqrt{(x-y)^2+(t-s)^2} -\ln \xi_0  \Bigr]   u(y,s) dy ds=\mu_0 h, \quad (x,t) \in \Omega,
\]
in which  $\Omega \subset \mathbb{R}^2$ is the cross-sectional domain, $h$ denotes  the conductivity of the material of the bar, $\nu$ signifies  the angular frequency of the alternating current, the constant $\xi_0$ indicates  the origin potential, $\tau$  represents  the permeability of empty space, and the constant $\mu_0$
is contingent upon the choice of  $\xi_0$ and has no physical importance.

 Analytically solving WSFIEs, especially in high dimensions, is typically challenging.   As a result, there has been significant attention given to the development of simple and effective numerical approaches for approximating these equations in the past few years.
   Recently, many authors have  paid much attention to study the numerical approaches in finding the solution of WSFIEs.
   The Galerkin and collocation approaches \cite{8,A1} are effective techniques for finding the solutions of integral equations (IEs) by utilizing a set of basis functions. The piecewise polynomial collocation and Galerkin approaches 
   \cite{A2,A3}, high-order collocation approaches  \cite{A5},  Bubnov-Galerkin schemes \cite{A6}, iterated fast multiscale Galerkin methods \cite{A7}, hybrid collocation methods \cite{A8}, sinc-collocation methods \cite{A9} and   discrete Petrov-Galerkin methods \cite{A10} have been employed for solving WSFIEs. Authors of \cite{A11,A12} have investigated Adomian decomposition methods for solving WSFIEs. 
   Also, Daubechies interval wavelets \cite{A13}, Legendre wavelets \cite{A14} and  Haar wavelets \cite{A15}  have been employed to solve WSFIEs.

      Kernel-based approximation techniques have become effective and valuable computational tools in a wide range of applications, including interpolation,  data modeling, multivariate integration,  multivariate optimization, approximation theory, statistics, and numerical solution of PDEs in multivariate spaces. One highly effective approach in literature to address such problems involves utilizing radial kernels. Over the past few years, meshless approximations applying radial basis functions (RBFs) in the standard  format have been widely adopted  to approximate a wide range of IEs.  For example,  radial kernels in conventional framework  \cite{A16} have been put forward to solve system of IEs of the second kind. The  {\color{red} MPI (meshless product integration) }  approach \cite{A17} has been applied  to approximate WSFIEs. 
      Radial kernels have been used in meshless discrete collocation methods to  approximate  both linear and nonlinear 2D-FIEs on irregular regions with weakly singular kernels \cite{10} as well as sufficiently smooth kernels \cite{11, 12}.
   In \cite{13}, a  meshless collocation technique based on RBFs has been proposed for singular-logarithmic boundary IEs. The RBF approach \cite{14, 15, 16, 17} has been employed for solving multivariate linear FIEs on the regular and irregular regions. A RBF-based meshless procedure \cite{18} has been applied to solve oscillatory FIEs. The papers \cite{19, 200} have outlined a stable meshless approach for solving  FIEs with sufficiently smooth kernels.
      
         These methods offer significant benefits in terms of their simplicity in implementation, adaptability to various geometries, and the ability to achieve spectral (or exponential) convergence rates by utilizing infinitely smooth radial kernels. The lowest possible error is typically attained when the shape parameter of the radial kernel is small. In other words, when the radial kernel is relatively flat. Nevertheless, when employing the conventional or  standard method of computing using infinitely smooth kernels, a contradiction arises between  stability and accuracy. In other words, as the shape parameter approaches zero, the standard operational matrix becomes considerably more ill-conditioned. The relationship between the  ill-conditioning of computational systems and the level of smoothness in the kernel has been extensively investigated in the literature over an extended period of time. There are several approaches available to address the issue of ill-conditioning. Up until now, a range of numerical algorithms have been presented to address the challenge of ill-conditioning when utilizing radial kernels.   A first stabilization method to work with multiquadric kernels in the flat limit was introduced by removing  the constraint that the shape parameter be real. \cite{220}. The technique utilized was referred to as the RBF-CP methodology and was  exclusively applicable to a limited quantity of nodes. 
         The authors of \cite{20, 21} presented a novel approach called RBF-QR algorithm for interpolation, which combines a QR decomposition of the interpolation matrix with series expansions of the kernel. This technique is specifically designed for interpolation using   Gaussian and zonal kernels.
         A stable method was introduced in \cite{22} by employing Mercer's theorem to acquire a new basis for the Gaussian interpolant. This interpolant is obtained from the eigenfunction expansion of the Gaussian kernel in multi-dimensional space. 

  Another effective technique to address ill-conditioning is by utilizing hybrid bases. Mishra and colleagues \cite{4300} initially presented the hybrid Gaussian-cubic radial kernel as a stable approach for solving interpolation problems associated with dispersed nodes.
  They demonstrated that incorporating a tiny portion of the cubic kernel into the Gaussian kernel effectively diminishes the condition number, thereby rendering the associated algorithm well-posed.
  Subsequent research \cite{43} has demonstrated that such a hybrid kernel  offers a sensible alternative for effectively solving partial differential equations using the RBF-PS method. 
In \cite{Yang1}, this kernel was utilized to reconstruct the temperature distribution within the measurement  field. Also, a novel approach was presented in \cite{omer}, which proposes a local hybrid Gaussian-cubic kernel approach to  approximate the 2-D fractional cable equation.
In  \cite{Manzoor}, the hybrid radial kernels  method was employed to solve the Burgers' equation, which featured varying Reynolds numbers.
Aside from the previously mentioned applications,  the hybrid kernels have also been used for detecting blockages within a fluid channel \cite{Cristhian}, the fractional Rayleigh-Stokes equations \cite{Nikan},   the gravity inversion approach \cite{Shuang}, the convection–diffusion equation \cite{Nissaya} and so on.

The primary concept behind this hybridization technique is to construct a kernel that combines the strengths of two distinct kernels, while also overcoming the limitations of each and preserving the formulation as a standard radial kernel methodology.
Other benefits of the suggested approach encompass the following:
\begin{itemize}
\item[$\bullet$] This technique is highly efficient, devoid of a mesh structure, and can easily adjust to irregular geometrical domains. Moreover, this approach is   unaffected by the dimensions of the problem.
\item[$\bullet$] The process of combining a kernel with infinite smoothness to a kernel with piecewise smoothness substantially enhances accuracy and diminishes the condition number.
\item[$\bullet$] The scope of this approach can be expanded to include other categories of IEs.
\item[$\bullet$] The application of this technique on computer systems is straightforward and highly appealing from a computational standpoint.
\end{itemize}

This article proposes  a numerical approach based on the hybrid bases for the solution of linear WSFIEs. The approach employs hybrid kernels built on dispersed nodes as a basis within the discrete collocation technique. This approach transforms  solving the problem under consideration into solving a  system of linear equations. Furthermore, based on the MAE, the optimal parameters in the hybrid kernels can be computed utilizing an enhanced PSO algorithm. We also study convergence of the proposed hybrid scheme.  Finally, we assess the accuracy and validity of the hybrid scheme by carrying out various numerical experiments. Also, the presented hybrid algorithm do not require a structured mesh, and therefore be applied to approximate complex geometry problems based on a set of dispersed nodes.
   
         The rest of the article unfolds as follows:   Section \ref{sec3}, delves into the fundamental formulations and characteristics of the HRKs technique. In Section \ref{sec4}, a meshless scheme utilizing hybrid kernels to solve the linear WSFIEs is introduced.
         {\color{red} The computational complexity of the scheme is investigated in Section \ref{sec63}.}
         In Section \ref{sec6}, the convergence of the proposed approach is discussed.
           In Section \ref{sec8}, we introduce  an enhanced PSO algorithm for determining the optimal  parameters in hybrid kernels when approximating the numerical solution of the WSFIEs.   Numerical examples are proposed  in Section \ref{sec7}. 
The article ends with concluding in Section \ref{sec9}.
\section{Fundamentals of Hybrid Radial Kernels: formulations and properties }\label{sec3}
In this part, we will examine key definitions and fundamental mathematical concepts regarding hybrid  kernels that are utilized in the present study.
\subsection{Approximation by HRKs}
Assume the set $\mathit{X}:=\{ \mathbfit{x}_i\}_{i=1}^{n}$ consists of a sequence of distinct  {\color{red} scattered} nodes over the region $\Omega\subset\mathbb{R}^d$.
A linear combination using the radial kernel $\Phi(\mathbfit{x}):=\phi(\Vert \mathbfit{x} \Vert)$, $\mathbfit{x} \in \mathbb{R}^d,$ is employed to estimate a function $u$ as belows \cite{44}:
\begin{align*}
u(\mathbfit{x})\approx\mathcal{P}_n u(\mathbfit{x}):=\sum_{j=1}^n\alpha_j \Phi_j(\mathbfit{x}),\quad x\in \Omega,
\end{align*}
in which
\begin{align*}
\Phi_j(\mathbfit{x}):=\phi(\Vert \mathbfit{x}-\mathbfit{x}_j\Vert)=\Phi \left( \mathbfit{x}-\mathbfit{x}_j \right).
\end{align*}
The coefficients $\mathbfit{\alpha}:=[\alpha_1,\alpha_2,...,\alpha_n]^T$ are calculated  by satisfying the interpolation conditions 
\begin{align*} 
\mathcal{P}_n u(\mathbfit{x}_i)=u(\mathbfit{x}_i) \quad i=1,2, \ldots, n,
\end{align*}
i.e. they can be acquired by solving a linear system 
\begin{equation}\label{eq11}
\mathbf{A}\mathbfit{ \alpha=}\mathbf{u},
\end{equation}
in which  $\mathbf{A}_{jk}:=[\phi(\Vert \mathbfit{x}_i-\mathbfit{x}_j\Vert)]_{i,j=1}^n$ and $\mathbf{u}:=[u(\mathbfit{x}_1),u(\mathbfit{x}_2),...,u(\mathbfit{x}_n)]^T$. There are two primary categories of radial kernels: piecewise smooth and  infinitely smooth.  Table \ref{tab1} lists a variety of well-known radial kernels.  All the infinitely smooth radial kernels presented in Table 1 will give coefficient matrices $\mathbf{A}$ in (\ref{eq11}) which are nonsingular. For  inverse multiquadrics (IMQ) and Gaussian (GA), the matrix is positive definite. For multiquadrics (MQ), there is one positive eigenvalue, while all other eigenvalues are negative, which makes the matrix invertible \cite{22, 31, 37, 39, 40, 42}. However, the piecewise smooth radial kernels mentioned in Table \ref{tab1} are symmetric and conditionally strictly  positive definite kernels. Thus, it is often necessary to use lower degree polynomials 
\[ \Pi(\mathbfit{x})=\sum_{k=1}^{m} \zeta_k p_k (\mathbfit{x}), \]
in order to guarantee invertibility.  Nonetheless,  in the majority of instances, radial kernels are utilized without the inclusion of augmented polynomials, leading to successful outcomes without facing singular matrices. Hence, the augmented polynomial terms are frequently eliminated.

Based on the aforementioned categorization, we can construct a hybrid radial kernel (HRK) family as belows:
\begin{align*}
\psi_j(\mathbfit{x}):=\alpha\Phi_j^\varepsilon(\mathbfit{x})+\beta\varphi_j(\mathbfit{x}), \quad j=1,2,...,n,
\end{align*}
in which $\Phi^\varepsilon(\mathbfit{x}):=\phi(\varepsilon\Vert \mathbfit{x}\Vert)$ is an infinitely smooth  radial  kernel with shape parameter $\varepsilon$ and $\varphi(\mathbfit{x}):=\varpi(\Vert \mathbfit{x} \Vert)$ is a piecewise smooth  radial  kernel free from shape parameter. 
The values  $\alpha$ and $\beta$ are positive real numbers that determine the influence of every kernel.
Moreover, it is important to note that scaling a kernel by a constant does not influence the outcome of the algorithm. \citep{43}. Hence, the just defined HRK can be normalized by utilizing the constant $\rho=\dfrac{\beta}{\alpha}$. This provides
\begin{equation}\label{eq13}
\psi_j^{\varepsilon,\rho}(\mathbfit{x})=\Phi_j^\varepsilon(\mathbfit{x})+\rho\varphi_j(\mathbfit{x}), \quad j=1,2,...,n.
\end{equation}
Now, the HRKs family (\ref{eq13})  includes two parameters, the shape parameter $\varepsilon$ and weight parameter  $\rho$, that control the accuracy and stability of the hybrid kernel approach.
We will now estimate the unknown function $u$  using the HRKs as belows:
\begin{align*}
u(\mathbfit{x})\approx\mathcal{Q}_nu(\mathbfit{x})=\sum_{j=1}^n c_j\psi_j^{\varepsilon, \rho}(\mathbfit{x}), \quad \mathbfit{x} \in \Omega\subset\mathbb{R}^d,
\end{align*}
where
\begin{align} \label{H11}
\psi_j^{\varepsilon, \rho}(\mathbfit{x}):=\Phi_j^\varepsilon(\mathbfit{x})+\rho\varphi_j(\mathbfit{x})=\phi_j(\varepsilon\Vert \mathbfit{x}-\mathbfit{x}_j\Vert)+\rho \varpi_j(\Vert \mathbfit{x}-\mathbfit{x}_j\Vert), \quad j=1,2,...,n.
\end{align}
The coefficients  $\mathbf{c}:=[c_1,c_2,...,c_n]^T$ are calculated via solving  $\mathbf{\Gamma c=u}$, where the hybrid  kernel matrix $\mathbf{\Gamma} \in \mathbb{R}^{n \times n}$ is 
\begin{align*}
\mathbf{\Gamma}_{i,j}=\psi_j^{\varepsilon,\rho}(\mathbfit{x}_i), \quad i,j=1, \ldots, n.
\end{align*}
\begin{table}
\centering
\caption{Some popular  radial kernels}\label{tab1}
\vspace*{0cm}
\begin{tabular}{lllllll}
\hline
\hline
\textbf{Name of radial kernel} &&  && \textbf{Definition }\\
\hline
{\textbf{Infinitely smooth:}} && && &&\\
\\
{Multiquadrics (MQ)} &&  && {$\sqrt{\varepsilon^2 r^2+1}$}\\
\\
{Inverse multiquadrics (IMQ)} &&  && {$\dfrac{1}{\sqrt{\varepsilon^2 r^2+1}}$}\\
\\
{Gaussian (GA)} &&  && {$e^{-(\varepsilon r)^2}$}\\
\\
\hline
{\textbf{Piecewise smooth:}} && && && \\
\\
{Thin plate spline (TPS)} &&  && {$ r^{2} \log(r)$}\\
\\
{Cubic (CU)}  &&  && {$r^{3}$}\\
\hline
\hline
\end{tabular}
\end{table}
\subsection{Error analysis for HRKs}
Here, we will delve into the error bound  of HRKs interpolation. In order to do so, it is necessary to introduce specific Hilbert spaces that are related to the radial kernels known as native Hilbert spaces.
\begin{theorem}\cite{31, 44}\label{theo1}
Let $\Phi\in L_1(\mathbb{R}^d) \cap C(\mathbb{R}^d) $ be a  strictly positive definite kernel. Then the native Hilbert {\color{red} space
with respect } to $\Phi$ is 
\begin{align*}
\mathcal{N}_\Phi(\mathbb{R}^d):=\bigg\lbrace h \in    L_2(\mathbb{R}^d) \cap C(\mathbb{R}^d):\dfrac{\hat{h}}{\sqrt{\hat{\Phi}}}\in L_2(\mathbb{R}^d)\bigg\rbrace,
\end{align*}
with inner product
\begin{align*}
\langle h,g \rangle_{\mathcal{N}_{\Phi(\mathbb{R}^d)}}:=(2\pi)^{-\frac{d}{2}}\langle\frac{\hat{h}}{\sqrt{\hat{\Phi}}},\frac{\hat{g}}{\sqrt{\hat{\Phi}}}\rangle_{L_{2(\mathbb{R}^d)}}=(2\pi)^{-\frac{d}{2}}\int_{\mathbb{R}^d}\frac{\hat{h}(w)\overline{\hat{g}(w)}}{\sqrt{\hat{\Phi}(w)}}dw,
\end{align*}
where $\hat{h}$ is the Fourier transform of $h$.
\end{theorem}
We can also extend the notion of native spaces for conditionally positive definite radial kernels, although the specific intricacies will not be discussed  in this study. {\color{red}For more information about the native Hilbert spaces,  refer to Theorem 10.12 on page 139 of reference \cite{ 44}.}

To obtain an error bound, we will now  present two common indicators of data regularity \cite{31, 44}:
\begin{definition}\label{def2}
 The   fill distance  of  $\mathit{X}=\{ \mathbfit{x}_i\}_{i=1}^{n} \subset \Omega$ is defined via
\begin{equation}\label{eq16}
h_{\mathit{X},\Omega}:=\sup_{\mathbf{\mathbfit{x}}\in \Omega}\min_{1\leq j \leq n}\Vert \mathbfit{x}-\mathbfit{x}_j\Vert_2.
\end{equation}
\end{definition}
\begin{definition}\label{def3}
The separation distance of $\mathit{X}=\{ \mathbfit{x}_i\}_{i=1}^{n}$ is defined via
\begin{equation}\label{eq17}
q_\mathit{X}:=\dfrac{1}{2}\min_{i\neq j}\Vert \mathbfit{x}_i-\mathbfit{x}_j \Vert.
\end{equation}
\end{definition}
\begin{remark}
The values  (\ref{eq16}) and (\ref{eq17}) describe an idea of the data distribution,  indicating the level of uniformity among nodes. In fact,  the set $\mathit{X}$ is called to be quasi-uniform if there is a constant $\vartheta >0$ such that
\[q_{\mathit{X}} \leq h_{\mathit{X},\Omega} \leq \vartheta q_{\mathit{X}}.\]
\end{remark}

From the definitions provided, we can now articulate the following theorem:
\begin{theorem}\label{theo2}\cite{44}
Let  $\Omega \subset \mathbb{R}^d $ be a bounded region that satisfies interior
cone condition. Let   $\Phi$ {\color{red} be a } conditionally positive deﬁnite kernel and $ \mathcal{P}_{n}u$  the radial kernel {\color{red} interpolating } function of  
$u \in\mathcal{N}_{\Phi}(\Omega)$ on the
set  $\mathit{X}=\{ \mathbfit{x}_i\}_{i=1}^{n}$.
Then, for some positive numbers  $h_0$ and $C$, we have 
\begin{align*}
\Vert u- \mathcal{P}_{n}u \Vert_{L^{\infty}(\Omega)} \leqslant C \sqrt{C_\Phi} h_{{}_{\mathit{X} , \Omega}}^{k}\Vert u \Vert _{\mathcal{N}_{\Phi}(\Omega)},
\end{align*}
provided  $h_{{}_{\mathit{X} , \Omega}} \leqslant h_0.$ The number  $C_\Phi$ is defined by
\[ C_\Phi= \max_{\vert \gamma \vert=2k} \max_{\varsigma , z \in \Omega \cap B(\mathbfit{x}, c_2h_{{}_{\mathit{X} , \Omega}})} \vert  \mathcal{D}_{2}^{\gamma}\Phi(\varsigma, z)  \vert, \]
\end{theorem}
{\color{red} where $\mathcal{D}_{2}^{\gamma}$   denotes the derivative of order $\gamma$ with respect to the second argument
and $B(\mathbfit{x}, c_2h_{{}_{\mathit{X} , \Omega}})$ is  the ball centered at $\mathbfit{x}$ and radius $c_2h_{{}_{\mathit{X} , \Omega}}$.} \\
The aforementioned theorem allows us to  describe an error bound for approximation achieved through hybrid kernels.
 {\color{red} \begin{remark}\label{cor101}
Note that   any  linear combination of conditionally positive definite kernels with nonnegative coefficients gives a conditionally positive definite \cite{31}.
\end{remark}
\begin{corollary}\label{cor1}
Let    $\Phi^{\varepsilon}( \mathbfit{x} )$  and $\varphi( \mathbfit{x})$  be conditionally  strictly positive definite kernels.  Then $\psi^{\varepsilon, \rho}(\mathbfit{x}):=\Phi^{\varepsilon}(\mathbfit{x})+\rho \varphi(\mathbfit{x})$ is conditionally positive deﬁnite for all $\rho>0.$
Also,  for the  hybrid kernel $\psi^{\varepsilon, \rho}(\mathbfit{x})  \in C^{2k}(\Omega \times \Omega) $, there exist numbers   $\hat{C}$ and $\tilde{h}_0$   such that
\begin{align*}
\Vert u- \mathcal{Q}_{n}u \Vert_{L^{\infty}(\Omega)} \leqslant \hat{C} \sqrt{\hat{C}_{\psi^{\varepsilon, \rho}}} h_{{}_{\mathit{X} , \Omega}}^{k}\Vert u \Vert _{\mathcal{N}_{{\psi^{\varepsilon, \rho}}}(\Omega)},
\end{align*}
where $\mathcal{N}_{{\psi^{\varepsilon, \rho}}}(\Omega)$ is the native space of kernel ${\psi^{\varepsilon, \rho}( \mathbfit{x} )},$ $\mathcal{Q}_{n}u$ is
the interpolating function of  $u \in\mathcal{N}_{{\psi^{\varepsilon, \rho}}}(\Omega)$, $h_{{}_{\mathit{X} , \Omega}} \leqslant \tilde{h}_0$ and
\[ \hat{C}_{\psi^{\varepsilon, \rho}}= \max_{\vert \gamma \vert=2k} \max_{\varsigma , z \in \Omega \cap B(\mathbfit{x}, c_2h_{{}_{\mathit{X} , \Omega}})} \vert  \mathcal{D}_{2}^{\gamma}{\psi^{\varepsilon, \rho}}(\varsigma, z)  \vert.  \]
\end{corollary}
\begin{proof}
The corollary follows immediately from Theorem \ref{theo2} and Remark \ref{cor101}.
\qed
 \end{proof}
}
\section{A meshless collocation scheme utilizing HRKs for linear WSFIEs }\label{sec4}
In this part, we presents an effective procedure utilizing the HRK approach to solve linear
WSFIEs of the second kind. 

\subsection{ One-dimensional  WSFIEs} \label{752}
We firstly discuss the method for one-dimensional WSFIEs. Consider the following one-dimensional WSFIE 
\begin{align}
u(x)- {\color{red}\lambda } \int_a ^b K(x ,t) u(t)dt=f(x), \quad  a\leq x,t \leq b, \label{equ61}
\end{align}
in which the rhs function $f$ and the kernel function
\begin{align*}
K(x,t)=R(x,t)S(x,t) 
\end{align*}
are given, $ \lambda \in \mathbb{R} \setminus \left\{ 0 \right\} $  and  $u$  is the unknown function. We also assume that $S$ is   an enough regular function on $[a,b]\times[a,b],$ and $R$ is a weakly singular function. To simplify our discussion, we can suppose that $a = 0$ and $b = 1$.

 To apply the method, we need the $n$ nodal points in the domain $[0, 1] $ as 
$0 \leqslant x_1 < x_2 < \cdots < x_n \leqslant 1$.
Thus, to solve equation (\ref{equ61}), we approximate $u(x)$   via the HRKs as
\begin{align} \label{DC1}
u(x) \approx \mathcal{Q}_{n} u(x) =\sum_{j=1}^{n} c_j \psi^{\varepsilon, \rho}_j(x), \quad x \in [0, 1],
\end{align}
{\color{red} in which  $ \psi^{\varepsilon, \rho}_j(x), j=1,\ldots,n,$ are defined in (\ref{H11}).}
Substituting  (\ref{DC1})  into  (\ref{equ61}) and then collocating at the points  $\{ x_i\}_{i=1}^{n}$
 we have 
 \begin{equation}\label{eq21300}
\sum_{j=1}^{n} c_j \Bigl( \psi^{\varepsilon, \rho}_j(x_i) -\lambda \int_{0}^{1} K (x_i, t)  \psi^{\varepsilon, \rho}_j(t) \Bigr) dt = f(x_i), \quad i=1,\ldots,n.
\end{equation}
Due to the singular nature of the integrals presented in equation (\ref{eq21300}), it is not possible to calculate them analytically. Additionally, they cannot be approximated using conventional numerical integration methods.
Here, we propose a straightforward yet effective integration rule from  \cite{45,46} that facilitates the calculation of these integrals.

We now {\color{red} recall } a  $m$-point  composite Gauss-Legendre (CGL) integration formula  for weakly singular integrals with $L$ non-uniform subdivisions. Assume that $g$ is  defined on $(0,1)$ and near $y=0$ {\color{red} satisfies }
\begin{align}\label{equ9}
\vert g^{(2m)}(y)\vert\leqslant \mathcal{C} y^{-\sigma-2m}, \quad \text{for some}\quad \sigma\in(0,1),
\end{align}
and for every $y \in(0,1)$. Also, assume {\color{red}  $\{\theta_k\}_{k=1}^{m}$  represent the $m$ zeros } of the Legendre polynomial of degree $m$ on $[-1,1]$ {\color{red} and $\{w_k\}_{k=1}^{m}$ } denote weights for Gauss-Legendre integration formula. Then, for each  integer $L>0$, we have \cite{45,46}
\begin{align}\label{equ10}
\int_0^1 g(y)dy=\sum_{q=1}^L \sum_{k=1}^{m} w_k \dfrac{\Delta h_q}{2}g(\theta_k^q)+\mathcal{O}\bigg(\dfrac{1}{L^{2m}}\bigg),
\end{align}
where
\begin{align*}
\Delta h_q:=h_q-h_{q-1},\quad \bar{h}_q:=\dfrac{h_q+h_{q-1}}{2},\quad \theta_k^q:=\dfrac{\Delta h_q}{2}\theta_k+\bar{h}_q,
\end{align*}
with $h_q:=\bigg(\dfrac{q}{L}\bigg)^s,$ \hspace*{0.05 cm}$s:=\bigg(\dfrac{2m+1}{1-\sigma}\bigg)$

To utilize the numerical integration method (\ref{equ10}) to the singular integrals in equation (\ref{eq21300}), we have
\begin{align*}
\int_{0}^{1} K (x_i, t)  \psi^{\varepsilon, \rho}_j(t) dt= \int_{0}^{x_i} K (x_i, t)  \psi^{\varepsilon, \rho}_j(t) dt +\int_{x_i}^{1} K (x_i, t)   \psi^{\varepsilon, \rho}_j(t) dt.
\end{align*}

By changing the variables $y= \frac{x_i-t}{x_i}$ and $y'= \frac{t-x_i}{1-x_i}$,
we have
\begin{align}\label{equ1800}
\int_{0}^{1} K (x_i, t)  \psi^{\varepsilon, \rho}_j(t) dt & =  x_i \int_{0}^{1} K \Big(x_i,  x_i(1-t) \Big) \psi^{\varepsilon, \rho}_j \Big(x_i(1-t)  \Big) dy \nonumber \\
&+ (1-x_i) \int_{0}^{1} K \Big(x_i,  t'(1-x_i)+x_i \Big) \psi^{\varepsilon, \rho}_j \Big(t'(1-x_i)+x_i \Big) dy'.
\end{align}

For each positive integer $m$  and each sufficiently small positive number $\sigma$, the condition stated in (\ref{equ9}) holds true for all integrals presented in equation (\ref{equ1800}).  Therefore, we can use the numerical quadrature formula (\ref{equ10}). So
\begin{align} \label{1n1}
\int_{0}^{1} K (x_i, t)  \psi^{\varepsilon, \rho}_j(t) dt \approx x_i \sum_{q=1}^{L} \sum_{k=1}^{m}  w_k \dfrac{\Delta h_q}{2}  K \Big(x_i,  x_i(1- \theta_k^q ) \Big) \psi^{\varepsilon, \rho}_j \Big(x_i(1- \theta_k^q )  \Big)
\nonumber \\ + (1-x_i) \sum_{q=1}^{L} \sum_{k=1}^{m}  w_k \dfrac{\Delta h_q}{2} K \Big(x_i,  \theta_k^q (1-x_i)+x_i \Big) \psi^{\varepsilon, \rho}_j \Big(\theta_k^q(1-x_i)+x_i \Big). 
\end{align}
Substituting integration scheme (\ref{1n1}) in system (\ref{eq21300}) yields 
\begin{align} \label{w33}
\sum_{j=1}^{n} \hat{c}_j \Bigg( \psi^{\varepsilon, \rho}_j(x_i) -\lambda x_i \sum_{q=1}^{L} \sum_{k=1}^{m}  w_k \dfrac{\Delta h_q}{2}  K \Big(x_i,  x_i(1- \theta_k^q ) \Big) \psi^{\varepsilon, \rho}_j \Big(x_i(1- \theta_k^q )  \Big)
\nonumber \\  -\lambda (1-x_i) \sum_{q=1}^{L} \sum_{k=1}^{m}  w_k \dfrac{\Delta h_q}{2} K \Big(x_i,  \theta_k^q (1-x_i)+x_i \Big) \psi^{\varepsilon, \rho}_j \Big(\theta_k^q(1-x_i)+x_i \Big) \Bigg) = f(x_i).
\end{align}
By solving linear system (\ref{w33}),  the approximate solution of equation
(\ref{equ61}) is computed as 
\begin{align*}
\hat{u}_{mn} (x) =\sum_{j=1}^{n} \hat{c}_j \psi^{\varepsilon, \rho}_j(x), \quad x \in [0, 1] .
\end{align*}

 \subsection{Two-dimensional WSFIEs} 

 We now apply the HRKs technique to solve the two-dimensional linear WSFIEs in the following form 
\begin{equation}\label{eq61}
u(x,t)-\lambda \int_{\Omega} K(x,t) u(y,s) dy ds=f(x,t), \quad (x,t) \in \Omega,
\end{equation}
in which the rhs function $f$ and the kernel function
\begin{align*}
K(x,t,y,s)=R(x,t,y,s)S(x,t,y,s)
\end{align*}
 are known, {\color{red} $ \lambda \in \mathbb{R} \setminus \left\{ 0 \right\}, $  } $u$  is the unknown function and $\Omega$ is a two-dimensional bounded region.
 Also, we  suppose that $S$ is an enough regular function on $\Omega \times \Omega$, and $R$ is a weakly singular function. 

To solve the  IE (\ref{eq61}), we choose $n$ arbitrary  {\color{red} scattered } nodes on the region $\Omega,$ such as $X=\bigl \{ (x_i, t_i)  \bigr \}_{i=1}^{n}.$ 
In order to use HRKs method, we  estimate $u(x,t)$ as 
\begin{equation}\label{eq211}
u(x,t)\approx \mathcal{Q}_{n} u(x,t) =\sum_{j=1}^{n} c_j \psi^{\varepsilon, \rho}_j(x,t), \quad (x,t) \in \Omega,
\end{equation}
where
\begin{align*} 
\psi^{\varepsilon, \rho}_j(x,t)
& = \Phi^{\varepsilon}_j(x,t)+\rho \varphi_j(x,t) \notag \\
& =\phi_j(\varepsilon\sqrt{(x-x_j)^2+(t-t_j)^2 })+\rho \varpi_j(\sqrt{(x-x_j)^2+(t-t_j)^2 }), \quad j=1, \ldots, n.
\end{align*}
Substituting  (\ref{eq211})  into  (\ref{eq61}) and then collocating at the points in  $X,$ 
 we have 
\begin{equation}\label{eq2300001}
\sum_{j=1}^{n} c_j \Bigl( \psi^{\varepsilon, \rho}_j(x_i,t_i) -\lambda \int_{\Omega} K (x_i, t_i, y, s) \psi^{\varepsilon, \rho}_j(y, s) \Bigr) dy ds = f(x_i, t_i).
\end{equation}
These integrals cannot be calculated analytically, so a special numerical technique of integration is needed.
Assume that
\begin{align*} 
\Omega=\{ (y,s) \in \mathbb{R}^2: a \leqslant y \leqslant b \hspace*{0.215cm} \textit{and} \hspace*{0.215cm} \alpha_{1}(y) \leqslant s \leqslant \alpha_{2}(y) \}, 
\end{align*} 
{\color{red} with }  $a,b \in \mathbb{R}$ and $\alpha_{1}$, $\alpha_{2}$   continuous functions.  To simplify our discussion, we can suppose that $a = 0$ and $b = 1$.
Let us take into account the function $g(y,s)$ defined on $\Omega$ which   {\color{red} satisfies }  the following condition for each positive integer  $m$ {\color{red} with a } small positive value $\mu$:
\begin{align*}
\bigg\vert\dfrac{\partial^{2m}g}{\partial y^{2m}}\bigg\vert<C_1y^{-\mu-2m}.
\end{align*}
Then applying the numerical quadrature formula (\ref{equ10}), we have
\begin{align*} 
\int_\Omega g(y,s)dyds=\int_0^1 \int_{\alpha_1(y)}^{\alpha_2(y)}g(y,s)dyds\simeq\sum_{q=1}^L\sum_{k=1}^{m}w_k\dfrac{\Delta y_q}{2}\int_{\alpha_1(\theta_k^q)}^{\alpha_2(\theta_k^q)}g(\theta_k^q,s)ds.
\end{align*} 
Because the integrand of $g(\theta_k^q,s)$ is a well-behaved function for all $s$, we can use the $m$-point CGL integration formula on $[\alpha_1(\theta_k^q),\alpha_2(\theta_k^q)]$ {\color{red} with equispaced points.} Drawing from the outlined scheme, we introduce the double $m$-point CGL integration formula with $L$ non-uniform subdivisions over the region $\Omega$. The following theorem provides an analysis of its error.
\begin{theorem}\label{theo1}
Assume that $g$ is defined over $\Omega \subseteq[0, 1]\times[0, 1]$ and fulfills \cite{45,46}
\begin{align}\label{equ16} 
\bigg\vert\dfrac{\partial^{2m}g}{\partial y^{2m}}\bigg\vert<C_1y^{-\mu-2m}, \quad \quad
\bigg\vert\dfrac{\partial^{2m}g}{\partial s^{2m}}\bigg\vert<C_2y^{-\mu},
\end{align}
for all $(y,s)\in \Omega$ and $\alpha_1, \alpha_2 \in C^{2m}[0, 1]$. Then, for any integer $L$, we have
\begin{align}\label{equ17}
\int_\Omega g(y,s)dyds=\sum_{q=1}^L\dfrac{\Delta y_q}{2}\sum_{k=1}^{m}w_k\dfrac{\Delta s(\theta_k^q)}{2}\sum_{r=1}^{L_{p,r}}\sum_{p=1}^{m}w_p g(\theta_k^q,\eta_p^r(\theta_k^q))+O\bigg(\dfrac{1}{L^{2m}}\bigg),
\end{align}
where
\begin{align*}
\theta_k^q=\dfrac{\Delta x_q}{2}v_k+\bar{x}_q, \quad \Delta x_q={x}_q-{x}_{q-1}, \quad \text{and} \quad \bar{x}_q=\dfrac{{x}_q+{x}_{q-1}}{2},
\end{align*}
with
\begin{align*}
{x}_q=\bigg(\dfrac{q}{L}\bigg)^s,\quad s=\dfrac{2q_n+1}{1-\mu}, \quad L_{p,r}=1+[L(\alpha_2(\theta_k^q)-\alpha_1(\theta_k^q))],
\end{align*}

\begin{align*}
\Delta s(\theta_k^q)=\dfrac{\alpha_2(\theta_k^q)-\alpha_1(\theta_k^q)}{L_{p,r}} \quad \text{and} \quad \eta_p^r(\theta_k^q)=\dfrac{\Delta s(\theta_k^q)}{2}s_p+\alpha_1(\theta_k^q)+\bigg(r-\dfrac{1}{2}\bigg)\Delta s(\theta_k^q).
\end{align*}
\end{theorem}

It is evident that the integrals in equation (\ref{eq2300001}) cannot be estimated using the integration formula (\ref{equ17}), as the singularity arises at the points $(x_j,t_j)$. In order to overcome this difficulty, we  divide  these integrals into  two distinct parts as follows:
\begin{align*}
I_i&=\int_\Omega K \big(x_i,t_i,y,s \big) \psi_j^{\varepsilon,\rho}(y,s) dsdy \\ &=\int_0^{x_i}\int_{\alpha_1(y)}^{\alpha_2(y)}K \big(x_i,t_i,y,s \big) \psi_j^{\varepsilon,\rho} \big(y,s \big)dsdy 
+\int_{x_i}^1\int_{\alpha_1(y)}^{\alpha_2(y)}K \big(x_i,t_i,y,s \big) \psi_j^{\varepsilon,\rho} \big(y,s \big) dsdy.
\end{align*}
Now, let
\begin{align*}
y=(1-\gamma)x_i \quad \text{and} \quad y=x_i+(1-x_i)\xi .
\end{align*}
With these  changes of variables, we have
\begin{align*}
I_i&=x_i\underbrace{\int_0^1\int_{\alpha_1((1-\gamma)x_i )}^{\alpha_2((1-\gamma)x_i )}}_{\omega_1}K_1 \bigg(x_i,t_i,(1-\gamma)x_i ,s \bigg) \psi_j^{\varepsilon,\rho} \bigg((1-\gamma)x_i ,s \bigg) d\gamma ds \\
&+(1-x_i)\underbrace{\int_0^1\int_{\alpha_1(x_i+(1-x_i)\xi)}^{\alpha_2(x_i+(1-x_i)\xi )}}_{\omega_2}K_2 \bigg(x_j,t_j, x_i+(1-x_i)\xi,s \bigg) \psi_j^{\varepsilon,\rho} \bigg(x_i+(1-x_i)\xi ,s \bigg)d\xi ds.
\end{align*}
It is important to point out that $K_1$ and $K_2$ exhibit singularities at  $s = t_i, \gamma=0$ and $ s = t_i, \xi=0$, respectively. 
Therefore, these functions fulfill the conditions outlined in (\ref{equ16}) for every positive integer $m$ and for every sufficiently small positive value of $\mu$.
Consequently, by applying the quadrature formula (\ref{equ17}) on $\omega_1$ and $\omega_2$, we have

\begin{align*} 
I_i &\approx x_i\sum_{q=1}^L\dfrac{\Delta y_q}{2}\sum_{k=1}^{m}w_k\dfrac{\Delta s_1(\theta_k^q)}{2}\sum_{r=1}^{L_{p,r,1}}\sum_{p=1}^{m}w_p K_1\bigg(x_i,t_i, (1-\theta_k^q) x_i,\eta_{p,1}^r(\theta_k^q) \bigg) \psi_j^{\varepsilon,\rho}  \bigg((1-\theta_k^q) x_i,\eta_{p,1}^r(\theta_k^q)  \bigg) \nonumber \\
&+(1-x_i)\sum_{q=1}^L\dfrac{\Delta y_q}{2}\sum_{k=1}^{m}w_k\dfrac{\Delta s_2(\theta_k^q)}{2}\sum_{r=1}^{L_{p,r,2}}\sum_{p=1}^{m}w_p K_2\bigg(x_i,t_i, x_i+ (1-x_i)\theta_k^q ,\eta_{p,2}^r(\theta_k^q) \bigg) \psi_j^{\varepsilon,\rho}  \bigg(x_i+ (1-x_i)\theta_k^q,\eta_{p,2}^r(\theta_k^q)  \bigg).
\end{align*} 

As a result,  the linear system (\ref{eq2300001}) {\color{red} becomes }
\begin{footnotesize}
\begin{align*}
&\sum_{j=1}^n\hat{c}_j \Bigg( \psi_j^{\varepsilon,\rho}(x_i,t_i)-\lambda x_i\sum_{q=1}^L\dfrac{\Delta y_q}{2}\sum_{k=1}^{m}w_k\dfrac{\Delta s_1(\theta_k^q)}{2}\sum_{r=1}^{L_{p,r,1}}\sum_{p=1}^{m}w_p K_1\bigg(x_i,t_i, (1-\theta_k^q) x_i,\eta_{p,1}^r(\theta_k^q) \bigg) \psi_j^{\varepsilon,\rho}  \bigg((1-\theta_k^q) x_i,\eta_{p,1}^r(\theta_k^q)  \bigg) \nonumber\\
&- \lambda(1-x_i)\sum_{q=1}^L\dfrac{\Delta y_q}{2}\sum_{k=1}^{m}w_k\dfrac{\Delta s_2(\theta_k^q)}{2}\sum_{r=1}^{L_{p,r,2}}\sum_{p=1}^{m}w_p K_2\bigg(x_i,t_i, x_i+ (1-x_i)\theta_k^q ,\eta_{p,2}^r(\theta_k^q) \bigg) \psi_j^{\varepsilon,\rho}  \bigg(x_i+ (1-x_i)\theta_k^q,\eta_{p,2}^r(\theta_k^q)  \bigg) \Bigg)=f(x_i,t_i).
\end{align*}
\end{footnotesize}

After solving {\color{red} it,} the approximate solution is calculated as 
\begin{align*} 
\hat{u}_{mn}(x,t)=\sum_{j=1}^n\hat{c}_j \psi_j^{\varepsilon,\rho}(x,t), \quad (x,t)\in \Omega.
\end{align*} 
{\color{red} \begin{remark} 
 Let the domain $\Omega$ be described as 
 \begin{align*} 
\Omega=\{ (y,s) \in \mathbb{R}^2: c \leqslant s \leqslant d \hspace*{0.215cm} \textit{and} \hspace*{0.215cm} \Upsilon_{1}(s) \leqslant y \leqslant \Upsilon_{2}(s) \}, 
\end{align*} 
 with $c,d \in \mathbb{R}$ and $\Upsilon_{1}$, $\Upsilon_{2}$   continuous functions.
 In this case, the computations are similarly performed,
but the variables are commuted. 
 \end{remark} }
\subsection{High-dimensional  WSFIEs  }
Here, we expand  the previous approach to approximate  high-dimensional WSFIE \ref{eq1}. 

In order to accomplish this, we require  $n$   {\color{red} scattered}
nodes $\mathit{X}=\{ \mathbfit{x}_i\}_{i=1}^{n}=\bigl \{ (x_{1_{i}}, \ldots, x_{d_{i}})  \bigr \}_{i=1}^{n}$ on the region $\Omega.$ 
The unknown function $u$ in equation  (\ref{eq1}) can be approximated via  HRKs   as
\begin{align*} 
u(x_1,  \ldots, x_d)\approx   \mathcal{Q}_{n} u(x_1,  \ldots, x_d) =\sum_{j=1}^{n} c_j \psi^{\varepsilon, \rho}_j(x_1,  \ldots, x_d), \quad (x_1,  \ldots, x_d) \in \Omega \subset \mathbb{R}^d,
\end{align*} 
where 
\begin{align}\label{eq22}
\psi^{\varepsilon, \rho}_j(x_1,  \ldots, x_d)
& = \Phi^{\varepsilon}_j(x_1, \ldots, x_d)+\rho \varphi_j(x_1, \ldots, x_d) \notag \\
& =\phi_j \Bigg(\varepsilon \sqrt{ \sum_{i=1}^{d} (x_i -x_{{i}_{j}})^2  } \Bigg)+\rho \varpi_j \Bigg(\sqrt{ \sum_{i=1}^{d} (x_i -x_{{i}_{j}})^2  } \Bigg), \quad  j=1, \ldots, n.
\end{align}
Then, from equations (\ref{eq22}) and (\ref{eq1}) we have
 \begin{align*} 
 \sum_{j=1}^{n} c_j  \Bigg( \psi^{\varepsilon, \rho}_j \Bigl(x_1,  \ldots, x_d \Bigr) -\lambda \int_{\Omega} K \Bigl(x_1,  \ldots, x_d, t_1,  \ldots, t_d \Bigr)  \psi^{\varepsilon, \rho}_j \Bigl(t_1,  \ldots, t_d \Bigr) \Bigg) dt_1 \ldots dt_d= f(x_1,  \ldots, x_d).
\end{align*} 
Now, we pick collocation points $\bigl \{ (x_{1_{i}}, \ldots, x_{d_{i}})  \bigr \}_{i=1}^{n}$  to transform this into a system of equations 
\begin{equation}\label{eq230}
  \sum_{j=1}^{n} c_j \Bigg( \psi^{\varepsilon, \rho}_j \Bigl(x_{{1}_{i}}, \ldots, x_{{d}_{i}} \Bigr) -\lambda \int_{\Omega} K \Bigl(x_{{1}_{i}}, \ldots, x_{{d}_{i}}, t_1,  \ldots, t_d \Bigr) \psi^{\varepsilon, \rho}_j \Bigl(t_1,  \ldots, t_d \Bigr) \Bigg) dt_1 \ldots dt_d= f(x_{{1}_{i}}, \ldots, x_{{d}_{i}}).
\end{equation}
Due to the singular nature of the integrals in equation (\ref{eq230}) at the point $(x_{{1}_{i}}, \ldots, x_{{d}_{i}})$, it is not possible to calculate them using classical numerical integration methods.
In order to approximate the integrals  {\color{red} appearing in } (\ref{eq230}), we  employ the  $m$-point CGL integration rule with $L$ non-uniform subdivisions  over $\Omega$ relative to  the quadrature nodes $ \mathbfit{t}_k = ( t_{k_1}, \ldots, t_{k_d} ) $  weights $ w_k$.
 It is supposed that \cite{10}
  \begin{footnotesize}
\begin{equation}\label{eq25}
\int_{\Omega} K \Bigl(x_{{1}_{i}}, \ldots, x_{{d}_{i}}, t_1,  \ldots, t_d \Bigr)    \psi^{\varepsilon, \rho}_j \Bigl(t_1,  \ldots, t_d \Bigr)   d\mathbfit{t}  \approx      \sum_{q_1=1}^{L} \sum_{k_1=1}^{m} \ldots \sum_{q_d=1}^{L} \sum_{k_d=1}^{m} \tilde{\mathbf{w}}_k K \Bigl(x_{{1}_{i}}, \ldots, x_{{d}_{i}}, \eta_{k_1}^{q_1}, \ldots, \eta_{k_d}^{q_d} \Bigr) \psi^{\varepsilon, \rho}_j \Bigl(\eta_{k_1}^{q_1}, \ldots, \eta_{k_d}^{q_d} \Bigr), 
\end{equation}
 \end{footnotesize}
where $\tilde{\mathbf{w}}_k$  and $ (\eta_{k_1}^{q_1}, \ldots, \eta_{k_d}^{q_d})  $
are calculated  via  the weights $w_k$ and the  points $  (t_1,  \ldots, t_d)$.
Using integration scheme (\ref{eq25}) in  (\ref{eq230}), {\color{red} it holds}
 \begin{footnotesize}
 \begin{equation}\label{eq240}
  \sum_{j=1}^{n} \hat{c}_j \Biggl( \psi^{\varepsilon, \rho}_j \Bigl(x_{{1}_{i}}, \ldots, x_{{d}_{i}} \Bigr) -\lambda \sum_{q_1=1}^{L} \sum_{k_1=1}^{m} \ldots \sum_{q_d=1}^{L} \sum_{k_d=1}^{m} \tilde{\mathbf{w}}_k K \Bigl(x_{{1}_{i}}, \ldots, x_{{d}_{i}}, \eta_{k_1}^{q_1}, \ldots, \eta_{k_d}^{q_d} \Bigr) \psi^{\varepsilon, \rho}_j \Bigl(\eta_{k_1}^{q_1}, \ldots, \eta_{k_d}^{q_d} \Bigr) \Biggr)= f(x_{{1}_{i}}, \ldots, x_{{d}_{i}}).
\end{equation}
 \end{footnotesize}
This provides a {\color{red}system of linear equations.}
Once the coefficients $\{\hat{c}_1, \hat{c}_2, \ldots, \hat{c}_n \}$ are determined, the approximate solution can be derived by 
\begin{align*}
 \hat{u}_{mn} (x_1,  \ldots, x_d) =\sum_{j=1}^{n} \hat{c}_j \psi^{\varepsilon, \rho}_j(x_1,  \ldots, x_d), \quad (x_1,  \ldots, x_d ) \in \Omega .
\end{align*}

 {\color{red}
 \section{Computational complexity} \label{sec63}
 Here, we investigate the computational complexity of the scheme to measure the usage time of the algorithm presented in the current work. A simplified analysis can be performed to estimate the number of primitive operation by counting the steps of pseudo-code or the statements of high level language executed \cite{461}.
 To this aim, we consider a pseudo-code description of the method for
solving WSFIEs in Algorithm 1. The computational complexity of the algorithm discussed in this paper is determined using ``Big-Oh'' notation which is defined as follows: 
 \begin{definition}
 Let $f(n)$ and $g(n)$ be functions mapping nonnegative integers to real numbers.
We say that $f(n)$ is $\mathcal{O}(g(n))$ if there is a real constant $c>0$ and an integer constant $n_0 \geqslant 1$ such that $f(n) \leqslant cg(n)$  for every integer $n \geqslant n_0$. This definition is referred to as `Big-Oh' notation.
Also, we can also say `$f(n)$ is order $g(n)$'. 
 \end{definition}
 It is easy to see that the maximum number of operations appears in the for-loop between the lines 21-23 in Algorithm 1 and the computational complexity in these lines overcomes other steps of the algorithm. Therefore, we can only compute the number of primitive operations in the mentioned for-loop to give the computational complexity which is 
\[   \mathcal{O}(nF(n, m, L, d)),  \]
where $F(n, m, L,d)$ is the cost of calculating the value $T(x_1,  \ldots, x_d)$. 
From lines 6-20 of Algorithm 1, we find that the number of operations for computing $T(x_1,  \ldots, x_d)$ is
 \[ \mathcal{O}(nQ(m, d)), \]
 where $Q(m, d)$ is the cost of calculating to estimate $I_j(x_1,  \ldots, x_d)$ which is of order $m^dL^d.$ 
Finally, it results that the complexity of proposed algorithm is 
\[   \mathcal{O}(n^2m^dL^d).  \]

}

\begin{algorithm}      {\color{red}
  \caption{\quad A brief overview of the proposed method }
  \hspace*{\algorithmicindent} \textbf{Input:} \text{ Dimension of problem: $d$ } \\   \hspace*{1.75cm} \text{ Number of points: $n$  }
    \\   \hspace*{1.75cm} \text{ Number of integration points: $m$ }
    \\   \hspace*{1.75cm} \text{ Number of integration  subdivisions: $L$ }
  \\   \hspace*{1.75cm} \text{  Scattered points:  $\bigl \{ (x_{1_{i}}, \ldots, x_{d_{i}})  \bigr \}_{i=1}^{n}$ } 
            \begin{algorithmic}[1]
    \State \text{$S(x_1,  \ldots, x_d) \leftarrow 0$ } \\
     \textbf{for} { $j=1$ \textbf{to} $n$ \textbf{do}}
     \State $S(x_1,  \ldots, x_d) \leftarrow S(x_1,  \ldots, x_d)+ \hat{c}_j \psi_j^{\varepsilon, \rho}(x_1,  \ldots, x_d) $ \\  
    \textbf{end for} \\
     $T(x_1,  \ldots, x_d) \leftarrow  0$ \\
     \textbf{for} { $j=1$ \textbf{to} $n$ \textbf{do}} \\
        $I_j(x_1,  \ldots, x_d) \leftarrow  0$ \\
     \quad  \textbf{for} { $q_1=1$ \textbf{to} $L$ \textbf{do}}\\
  \quad    \quad  \textbf{for} { $k_1=1$ \textbf{to} $m$ \textbf{do}}\\
 \quad \quad      \quad { $\vdots$} \\
  \quad \quad  \quad  \textbf{for} { $q_d=1$ \textbf{to} $L$ \textbf{do}}\\
  \quad \quad  \quad    \quad  \textbf{for} { $k_d=1$ \textbf{to} $m$ \textbf{do}}\\
    \quad \quad  \quad  \quad  \quad  $I_j(x_1,  \ldots, x_d) \leftarrow  I_j(x_1,  \ldots, x_d)+\tilde{\mathbf{w}}_k K \Bigl(x_1,  \ldots, x_d, \eta_{k_1}^{q_1}, \ldots, \eta_{k_d}^{q_d} \Bigr) \psi^{\varepsilon, \rho}_j \Bigl(\eta_{k_1}^{q_1}, \ldots, \eta_{k_d}^{q_d} \Bigr)$
   \\  
     \quad \quad  \quad    \quad \textbf{end for} \\
 \quad  \quad    \quad \textbf{end for} \\
  \quad \quad      \quad { $\vdots$} \\
  \quad \quad  \textbf{end for} \\
   \quad  \textbf{end for} 
   \State $T(x_1,  \ldots, x_d) \leftarrow T(x_1,  \ldots, x_d)+ \hat{c}_j I_j(x_1,  \ldots, x_d)$ \\
   \textbf{end for} \\
      \textbf{for}  { $i=1$ \textbf{to} $n$ \textbf{do}}
        \State $F_i \leftarrow S(x_{{1}_{i}}, \ldots, x_{{d}_{i}})-\lambda T(x_{{1}_{i}}, \ldots, x_{{d}_{i}})=f(x_{{1}_{i}}, \ldots, x_{{d}_{i}}) $ \\
    \textbf{end for} 
       \State \textbf{solve $\{F_1, \ldots, F_n \}$}
          \State $\hat{u}_{mn}(x_1,  \ldots, x_d) \leftarrow $  \textbf{assign } $\{ \hat{c}_1, \ldots, \hat{c}_n \} $ \textbf{in }      \text{$S(x_1,  \ldots, x_d) $ }
    \Statex \textbf{Output:}  $\hat{u}_{mn}(x_1,  \ldots, x_d)$ 
  \end{algorithmic}
}
\end{algorithm}

\section{Convergence analysis of the proposed Hybrid Kernel Method} \label{sec6}
 In compact form, we rewrite the WSIE (\ref{eq1}) as  
  \begin{align*} 
 (I- \lambda \mathcal{F})u=f,
 \end{align*} 
in which the linear integral operator $\mathcal{F}:C(\Omega)\longrightarrow C(\Omega)$   is defined as 

  \begin{align}\label{equ33}
 \mathcal{F}u(\mathbfit{x})=\int_{\Omega} \mathcal{R}(\mathbfit{x}, \mathbfit{t}) \mathcal{S}(\mathbfit{x}, \mathbfit{t})u(\mathbfit{t}) d\mathbfit{t}.
 \end{align}
 We will now explore the concept of compact operators along with several significant theorems related to them.
 \begin{definition} \cite{8} \label{def2}
 A linear operator $\mathcal{F}: C(\Omega) \longrightarrow C(\Omega)$ {\color{red} is said to be } compact if the set $\{\mathcal{F}u: \Vert u\Vert \leqslant 1\}$ has compact closure in $C(\Omega)$.
 \end{definition}
 
 \begin{theorem} \cite{A1} \label{theo2222}
 Compact linear operators are bounded.
 \end{theorem}
 
  \begin{theorem} \cite{A1} \label{theo3}
 The weakly singular integral operator  (\ref{equ33}) is a compact operator on $C(\Omega)$.
 \end{theorem}
 
 Now applying the double $m$-point CGL integration scheme with $L$   subdivisions over the region $\Omega$ {\color{red} given in } (\ref{eq25}), we define a sequence of numerical weakly singular  integral operators $\mathcal{F}_m$ on $C^{(2m)}(\Omega)$ via 
\begin{align*} 
\mathcal{F}_mu(\mathbf{x}):=\lambda \sum_{k=1}^{m} \tilde{\mathbf{w}}_k K \Bigl(x_{{1}_{i}}, \ldots, x_{{d}_{i}}, \eta_{1_k}^{q}, \ldots, \eta_{d_k}^{q} \Bigr)  \psi^{\varepsilon, \rho}_j \Bigr(\eta_{1_k}^{q}, \ldots, \eta_{d_k}^{q}\Bigr), \qquad m \geq 1.
\end{align*} 
 It is crucial to note that{\color{red} $\{\mathcal{F}_m \}$ is a succession of collectively compact operators } that converges pointwise \cite{45, 46}. For each $u \in C^{(2m)}(\Omega)$ with a weakly singular kernel  $K$, we have \cite{45} 
\begin{equation}\label{eq31}
\Vert \mathcal{F}-\mathcal{F}_m \Vert \leq \frac{C_n}{L^{2m}} 
\sup_{\mathbfit{x} \in \Omega} \vert u^{(2m)}(\mathbfit{x}) \vert,
\end{equation}
where $C_n >0$ is constant.
 
 Consequently, the ultimate system  (\ref{eq240}) can be viewed as 
  \begin{align*} 
  (I-\lambda\mathcal{Q}_{n}\mathcal{F}_m)\hat{u}_{mn}=\mathcal{Q}_{n}f,
  \end{align*} 
  where $\hat{u}_{mn}$ is the numerical solution of the IE (\ref{eq1}) applying the {\color{red} proposed} technique. \\
The iterated discrete collocation solution $\bar{u}_{mn}$  is introduced by
  \begin{align}\label{equ41}
  \bar{u}_{mn}=f+\lambda\mathcal{Q}_{n}\mathcal{F}_m\hat{u}_{mn},
  \end{align}
  then it is easily seen that
  \begin{align*} 
  \mathcal{Q}_{n} \bar{u}_{mn}=\hat{u}_{mn},
  \end{align*} 
  and so
  \begin{align*} 
  (I-\lambda\mathcal{Q}_{n}\mathcal{F}_m)\bar{u}_{mn}=\mathcal{Q}_{n}f.
  \end{align*} 
We present the  theorem below, which addresses the error analysis of iterated collocation scheme \citep{8}.

    \begin{theorem}\label{theo5}
Let $\{\mathcal{Q}_n\}$  be a family of interpolatory projection operators on $C(\Omega)$ and suppose
\[   \mathcal{Q}_n u\longrightarrow u \quad  \textit{as } \quad n\longrightarrow\infty,    \]
for all $u \in  C(\Omega)$ and $u^\star$ be a unique solution of equation (\ref{eq1}). 
Then, $(I-\lambda\mathcal{F}_m\mathcal{Q}_n)^{-1}$ exists for all  $m, n$ sufficiently large, say $m, n >  \mathfrak{T}$ , and is uniformly bounded. Additionally, {\color{red} it holds}
 \begin{align*} 
  \Vert \bar{u}_{mn} - u^\star \Vert_{L^{\infty}(\Omega)}\leqslant \Vert (I-\lambda\mathcal{F}_m\mathcal{Q}_n)^{-1} \Vert \Vert \mathcal{F}u^\star -\mathcal{F}_m \mathcal{Q}_n u^\star\Vert_{L^{\infty}(\Omega)}, \quad m, n \geqslant  \mathfrak{T}.
  \end{align*}

 \end{theorem}
{\color{red}We now prove the convergence of the proposed method.  }
 \begin{theorem}\label{theo7}
 Consider the assumptions   of Theorem \ref{theo5} be fulfilled.  
 Let $u^\star \in \mathcal{N}_{\psi^{\varepsilon, \rho}} (\Omega) \cap C^{2m} (\Omega)$ {\color{red} be the } unique  solution of the WSIE (\ref{eq1}).
 {\color{red} Let us suppose that a quasi-uniform set of nodes in $\Omega$ has been used. }
 Then, there exists $\mathcal{T}> 0$, such that for each $m, n> \mathcal{T}$, {\color{red}the linear system (\ref{eq240}) } has a unique solution $\hat{u}_{mn}$ which converges to $u^\star$ as $m, n\longrightarrow\infty $. Additionally, there exist {\color{red}constants } $\gamma_1, \gamma_2, \hat{C}, C_n, \tilde{h_0}$, such that
 \begin{align*} 
 \Vert \bar{u}_{mn} -u^\star\Vert _{L^{\infty}(\Omega)}\leqslant\gamma_1 \gamma_2 \hat{C} \sqrt{\hat{C}_{\psi^{\varepsilon, \rho}}} h_{{}_{\mathit{X} , \Omega}}^{k}\Vert u^\star \Vert _{\mathcal{N} _{{\psi^{\varepsilon, \rho}}}(\Omega)}+\dfrac{\gamma_1 C_n}{L^{2m}}\sup_{\text{x}\in \Omega}\vert {u^{\star}}^{(2m)}(\text{x})\vert, \quad m, n \geqslant \mathcal{T},
 \end{align*} 
 
 provided that $h_{X,\Omega}\leqslant \tilde{h_0}$.
 \end{theorem}

 \begin{proof}
 Clearly $h_{{}_{\mathit{X} , \Omega}} $ vanishes when $n\longrightarrow \infty$ due to the definition of quasi-uniform set.  Also, Corollary  \ref{cor1}  states that  for each $u^\star \in \mathcal{N} _{{\psi^{\varepsilon, \rho}}}(\Omega)$, there exist  positive numbers  $\hat{C}$ and $\tilde{h_0}$   such that 
\begin{equation}\label{eq40}
\Vert u^\star- \mathcal{Q}_{n}u^\star \Vert_{L^{\infty}(\Omega)} \leq \hat{C} \sqrt{\hat{C}_{\psi^{\varepsilon, \rho}}} h_{{}_{\mathit{X} , \Omega}}^{k}\Vert u^\star \Vert _{\mathcal{N} _{{\psi^{\varepsilon, \rho}}}(\Omega)},
\end{equation}
provided $h_{{}_{\mathit{X} , \Omega}} \leqslant \tilde{h_0}$ and thus $\mathcal{Q}_{n}u^\star \longrightarrow u^\star$ when $n\longrightarrow \infty.$ 
According to relation  (\ref{eq31})  for each $u^\star \in C^{(2m)}(\Omega),$
 we obtain
\begin{equation}\label{eq41}
\Vert \mathcal{F} u^\star-\mathcal{F}_m u^\star \Vert_{L^{\infty}(\Omega)} \leq \frac{C_n}{L^{2m}} \sup_{\mathbfit{x} \in \Omega} \vert {u^{\star}}^{(2m)}(\mathbfit{x}) \vert.
\end{equation}
{\color{red}Therefore, } $\mathcal{F}_m u^\star \longrightarrow \mathcal{F} u^\star$
as $m \longrightarrow \infty.$ Consequently, the hypothesis of {\color{red} Theorem} \ref{theo5}  are fulfilled. So, we conclude that the iterated scheme {\color{red} provides } a solution $\bar{u}_{mn}$ and there exists $ \mathfrak{T} > 0$ such that for each $m, n \geqslant \mathfrak{T} , (I-\lambda\mathcal{F}_m\mathcal{Q}_{n})^{-1}$ exists and
  \begin{align*}
  \Vert\bar{u}_{mn} - u^\star\Vert_{L^{\infty}(\Omega)}\leqslant\gamma_1\Vert \mathcal{F}u^\star -\mathcal{F}_m \mathcal{Q}_{n}u^\star\Vert_{L^{\infty}(\Omega)},
  \end{align*}
  where $\Vert  (I-\lambda\mathcal{F}_m\mathcal{Q}_{n})^{-1}\Vert\leqslant\gamma_1$. By considering $\hat{u}_{mn}=\mathcal{Q}_n \bar{u}_{mn}$, we deduce that $\hat{u}_{mn}$ is {\color{red} a numerical solution provided by } the {\color{red} proposed } approach.  {\color{red} Applying }   $\mathcal{Q}_{n}$  {\color{red} to}  both sides of equation (\ref{equ41}) yields
  \begin{align*}\mathcal{Q}_{n}\bar{u}_{mn}=\mathcal{Q}_{n}f+\lambda\mathcal{Q}_{n}\mathcal{F}_m\hat{u}_{mn} \Rightarrow (I-\lambda\mathcal{Q}_{n}\mathcal{F}_m)\hat{u}_{mn}=\mathcal{Q}_{n}f.
  \end{align*}
  In addition, we have \cite{8}
  \begin{align*}
  (I-\lambda\mathcal{Q}_{n}\mathcal{F}_m)^{-1}=\lambda[I-\mathcal{Q}_{n}(I-\lambda\mathcal{F}_m\mathcal{Q}_{n})^{-1} \mathcal{F}_{m}], \quad \mathcal{T}\geqslant \mathfrak{T}.
  \end{align*}
As a result, the existence and boundedness of $(I-\lambda\mathcal{F}_m\mathcal{Q}_{n})^{-1}$ yield forthwith the existence and boundedness of $(I-\lambda\mathcal{Q}_{n}\mathcal{F}_m)^{-1}$. {\color{red}Then, } $\hat{u}_{mn}$ is {\color{red} the} unique solution {\color{red} provided by the proposed } approach.
According to the uniform boundedness principle \cite{8}, and
 the pointwise convergence of $\mathcal{F}_m$, we can suppose that $\Vert \mathcal{F}_m\Vert\leqslant\gamma_2<\infty$, so
  \begin{align*}
  \Vert \bar{u}_{mn} -u^\star\Vert_{L^{\infty}(\Omega)} & \leqslant\gamma_1\bigg[\Vert \mathcal{F}_m(u^\star-\mathcal{Q}_{n}u^\star)\Vert_{L^{\infty}(\Omega)}+\Vert\mathcal{F}u^\star -\mathcal{F}_m u^\star\Vert_{L^{\infty}(\Omega)}\bigg] \nonumber 
  \\
  & \leqslant\gamma_1 \gamma_2\Vert (u^\star-\mathcal{Q}_{n}u^\star)\Vert_{L^{\infty}(\Omega)}+\gamma_1\Vert\mathcal{F}u^\star -\mathcal{F}_m u^\star\Vert_{L^{\infty}(\Omega)}.
  \end{align*}
  Now, using (\ref{eq40}) and (\ref{eq41}), {\color{red} the claim follows. }
  \qed
 \end{proof}

\section{Enhanced Particle Swarm Optimization for optimal parameters in Hybrid Kernels }\label{sec8}
In this part, we introduce an innovative approach for determining the optimal parameters within the HRKs technique.

As discussed earlier, the hybrid kernels used in this study primarily involve combining an infinitely smooth kernel with one of finite smoothness. Infinitely smooth kernels often contain a shape parameter. Moreover, the creation of hybrid kernels also includes the incorporation of a weight parameter.  Hence, the hybrid kernels incorporated in this research consist of two distinct parameters, namely the shape parameter and the weight parameter. Proper selection of these parameters holds utmost importance as they heavily influence the accuracy and stability   of the approximations derived from kernel-based techniques.

There are numerous methods to optimize these parameters. Typically,  optimization methods can be classified into two primary categories: deterministic and stochastic methods. In addition, deterministic approaches can be categorized as computational approaches and gradient-based approaches \cite{Mirjalili, Khodadadi}. Computational techniques do not necessitate the computation of gradient functions, but are ineffective and slow. Gradient-based techniques utilize the derivative  of the cost function,  but they do not ensure reaching the global optimum when the cost function lacks smoothness. It can be challenging, or  even unattainable, to solve  multimodal, multivariate, or non-differentiable problems using deterministic approaches. Over the last several years, scientists have made notable progress in designing numerous metaheuristic algorithms inspired by nature. These methods are specifically designed to solve complex optimization problems that cannot be efficiently resolved using traditional methods within a satisfactory timeframe or with a high degree of accuracy. Metaheuristic algorithms are capable of finding the global optimum without relying on gradient information. They achieve this by iteratively invoking the cost function. These approaches restrict the scope of the search and aim to efficiently discover solutions \cite{ Khalilpourazari, Faris}. 
A study conducted by references \cite{Khodadadi, Deb} has presented a comprehensive analysis of nature-inspired metaheuristic algorithms and gradient-based algorithms, examining their respective strengths and weaknesses.

 Here, we introduce a  particle swarm optimization (PSO) approach to determine the optimal parameter values for the hybrid kernels.
   Prior to that, let us present the cost function.
\subsection{Cost Function}
In the current study, the selected cost function for optimizing parameters is derived from  the  maximum absolute  error (MAE).
The calculation of the  MAE over $M$ evaluation points is performed using the following relation:
\begin{align*}
\Vert e \Vert_{\infty}:= \max_{1 \leqslant j \leqslant M} \bigl \{ \vert u(\zeta_j)-\tilde{u}_{mn}(\zeta_j) \vert  \bigr \},
\end{align*}
where $\zeta_j$ denote the evaluation points and $\Vert e \Vert_{\infty}$ indicates  to the error function that needs  to be optimized for the minimum values for a set of  $\rho$ and  $\varepsilon$. Additionally,  $u$  represents the exact solution of the problem, while $\tilde{u}_{mn}$ signifies the solution obtained through the method outlined in the paper.

Therefore,  we can formulate the optimization problem as belows:
\begin{center}
Minimize  $O_f(\varepsilon, \rho)$ \\
subject to $\varepsilon, \rho \geq 0,$\\
\end{center}
in which   
\[    O_f(\varepsilon, \rho):= \Vert e \Vert_{\infty}(\varepsilon, \rho)= \max_{1 \leqslant j \leqslant M} \bigl \{ \vert u(\zeta_j)-\tilde{u}_{mn}(\zeta_j)(\varepsilon, \rho) \vert  \bigr \} \]
 denotes the cost  function   and $\varepsilon$  and  $\rho$ are the kernel parameters. 
 {\color{red}\begin{remark} 
Since the aim of the present paper is the comparison of the results of HRKs method for different WSFIEs, problems with existing exact solutions are considered. It must be emphasized that for practical problems where the exact solution is not known, it is not possible to calculate the exact MAE norm. In such situations, leave-one-out cross-validation (LOOCV) \cite{Rippa} is a prominent technique that can be used in conjunction with the PSO algorithm to find the optimal values of parameters.
\end{remark} }

Now, we are prepared to introduce an effective technique for selecting parameters of the HRKs using the PSO algorithm.
\subsection{Optimizing Parameters via PSO}

PSO is an efficient optimization method utilizing swarm intelligence to address various optimization problems. 
In 1995, Kennedy and Eberhart \cite{Kennedy} introduced the concept of a swarm, which is composed of particles that navigate through the search space of a given problem.
Every particle has a fitness value representing a potential solution. Particles remember their best positions (Pbest) and are influenced by the best value among all particles (Gbest). The PSO algorithm focuses on using these best positions to guide the swarm. In this study, PSO is employed to calculate the weight parameter and shape parameter for the hybrid kernels.
  

 Let $N_P$ be the number of the particles and  $\xi_j=[\varepsilon_j, \rho_j]$ be the candidate solution of particle $j$ $ (j = 1, 2, \ldots, N_P)$. In mathematical terms,  the velocity and the position of the  $j$-th particle  are calculated as
\begin{align*}
 \begin{cases}
v_j^{t+1}=wv_j^{t} +c_1 r_1({\color{red}{\operatorname*{Pbest}}_j } -\xi_j^{t})+c_2 r_2(\operatorname*{Gbest} -\xi_j^{t}), \\
\xi_j^{t+1}=\xi_j^{t}+ v_j^{t+1}, 
 \end{cases}   
\end{align*}
where $v_j^{t}$ is  the velocity vector of particle $j$ at iteration $t$, $\xi_j^{t}$ signifies  the  position vector of particle $j$ at iteration $t$,   $c_1$  and $c_2$   are the learning factors, $r_1$ and $r_2$  are uniform random numbers in the domain $(0,1)$,  $w$ indicates  the inertia weight factor, {\color{red}${\operatorname*{Pbest}}_j$ } denotes the Pbest of agent $j$ at iteration $t$,  and {\color{red} $\operatorname*{Gbest}$} represents the best solution  obtained by all particles so far. 

PSO offers several significant benefits, including easy implementation, having memory, mutual cooperation between particles and information sharing,  low complexity and  rapid convergence.    Nevertheless, the traditional PSO approach tends to prematurely converge and is prone to becoming ensnared in local optimal solutions.  In order to boost the effectiveness of the PSO algorithm, we have considered three significant modifications. These modifications can be summarized as follows:
\begin{itemize}
\item[$\bullet$]  \textbf{Initial population.}
In any evolutionary algorithms, population initialization  is  the first crucial task since it can significantly   increase the convergence speed and the quality of the ultimate solution to difficult problems. Overall,  the most frequent  approach for initializing the population  is to randomly generate initial swarm. In the literature on metaheuristics,  it can be perceived that a chaotic sequence rather than a randomly generated sequence for initialization is an effective method to improve the algorithm’s performance by avoiding premature convergence  \cite{Tian}. In this study, we employ  a chaotic PSO  (CPSO) algorithm in place of a random initial population position, serving as an innovative approach to forming the initial population.  In this approach, we use
the logistic map to produce chaotic variables and subsequently apply these variables within our search process.  The logistic equation can be given as follows:
\begin{equation}\label{s64}
z_{j+1}:=\mu z_{j}(1-z_{j}), \quad  j=1,2, \ldots
\end{equation}
where $z_j$  are dispersed in the range $(0,1)$, such that the initial $z_0 \in (0,1)$
and $z_0 \notin \{0, 0.25, 0.5, 0.75, 1\}$. $\mu$ is a fixed value known as the bifurcation coefficient. As the parameter $\mu$ gradually rises from zero, the system described by equation (\ref{s64}) transitions from a one fixed point to 2, then 3, and continues up to $2^j$. In this strategy, an increasing number of  multiple periodic components will be distributed across  smaller and smaller intervals of $\mu$ as it rises. This phenomenon is clearly  {\color{red} free } from restriction. However,  $\mu$ has a
limit value $\mu_c=3.569945671$ . 
Also, as $\mu \longrightarrow \mu_c$ , the period can turn infinite or even become non-periodic.
 In this case, the entire system converts into a chaotic state. Also,  the entire system becomes unstable when  $\mu$  is bigger than 4.
Therefore the interval $[\mu_c,4]$ is deemed as the chaotic domain of the entire system. 
In the following, we introduce  the  CPSO  algorithm for selecting  the optimal parameter values of the hybrid kernels. First, the logistic map is used to produce the chaotic variable. Therefore, equation  (\ref{s64}) is rewritten as follows:
\begin{equation}\label{s65}
 [ \varepsilon^{'}_{j+1} ,\rho^{'}_{j+1}] =\mu [\varepsilon^{'}_{j}(1-\varepsilon^{'}_{j}), \rho^{'}_{j}(1-\rho^{'}_{j}) ],
\end{equation}
where  $[\varepsilon^{'}_{j},\rho^{'}_{j}]$  is the chaotic variables vector at iteration $j$ and $\mu$ is  bifurcation coefficient.
 Now, set $j=0$ and produce two chaotic variables by equations  (\ref{s65}). 
Subsequently, set $j=1,2, \ldots, j$  in succession to produce the initial chaotic variables. 
 Lastly, incorporate these acquired variables into the problem's search domain as follows:
\begin{align*}
 \begin{cases}
 \varepsilon_{j}=\varepsilon_{\min}+\varepsilon^{'}_{j}(\varepsilon_{\max}-\varepsilon_{\min})  & {} \\
\rho_{j}=\rho_{\min}+\rho^{'}_{j}(\rho_{\max}-\rho_{\min}).    & {}
 \end{cases}   
\end{align*}
 As a result, 
\begin{align*}
\xi_j=[\varepsilon_j, \rho_j], \quad  j=1,2,  \ldots, N_P.
\end{align*}
That {\color{red} is how swarm} initialization based on the logistic map can be achieved.

\item[$\bullet$] \textbf{Inertia weight.}
 The role of inertia weight in the PSO methodology is  crucial. $w$  regulates the influence of a particle's prior velocity on its present velocity. 
 A suitable value for the inertia weight  $w$ typically establishes an equilibrium between local and global exploration abilities, thereby yielding a better optimal solution.
 In this work, the sine  map \cite{Chen} is applied to tune  inertia weights $w$ of the PSO algorithm throughout  the search procedure. The sine map utilized in the PSO approach not only can effectively improve population diversity during the search procedure but also can move to the global optimal solution more quickly.
 Therefore, the inertia weight $w$ can be defined as
 \begin{align*}
 w=x_t=\frac{ \varrho}{4} \times \sin(\pi x_{t-1}), \quad x_t \in (0,1), \quad 0< \varrho \leqslant 4, \quad t=1,2, \ldots, T, 
 \end{align*}
 in which $t$  signifies the current iterations and  $T$ indicates  the maximum iteration.

\item[$\bullet$] \textbf{Learning coefficients.}
 Within the PSO technique,  the learning coefficients  $c_1$ and $c_2$ are associated with the cognitive and social components respectively. These coefficients are known as stochastic acceleration factors, which play a crucial role in adjusting the particle velocity based on the Gbest and Pbest.
 Authors of \cite{Chen}  presented sine cosine acceleration coefficients (SCAC) as a novel method for adjusting parameters relating to cognitive and social components. In this strategy, the value of $c_1$ is between $ 2.5$ and $0.5$ and the value  of $c_2$ is between $0.5$ and $ 2.5$.  
This modification can be expressed as:
 \begin{align*}
 \begin{cases}
 c_1=2+0.5 \sin \Bigl( \Bigl( 1-\frac{t}{T}   \Bigr) \times \frac{\pi}{2} \Bigr),   & {} \\
c_2=2+0.5 \cos \Bigl( \Bigl( 1-\frac{t}{T}   \Bigr) \times \frac{\pi}{2} \Bigr),      & {}
 \end{cases}  \quad t=1,2, \ldots, T,
\end{align*}
in which $t$  signifies the current iteration and  $T$ indicates  the maximum iteration.

\end{itemize}

Figure \ref{a2000} illustrates the PSO method  utilized in this work.  {\color{red} In this figure, the stopping  criterion is set to the “no significant improvement observed for 20 iterations (i.e., similar fitness values achieved for 20 consecutive iterations)”. }

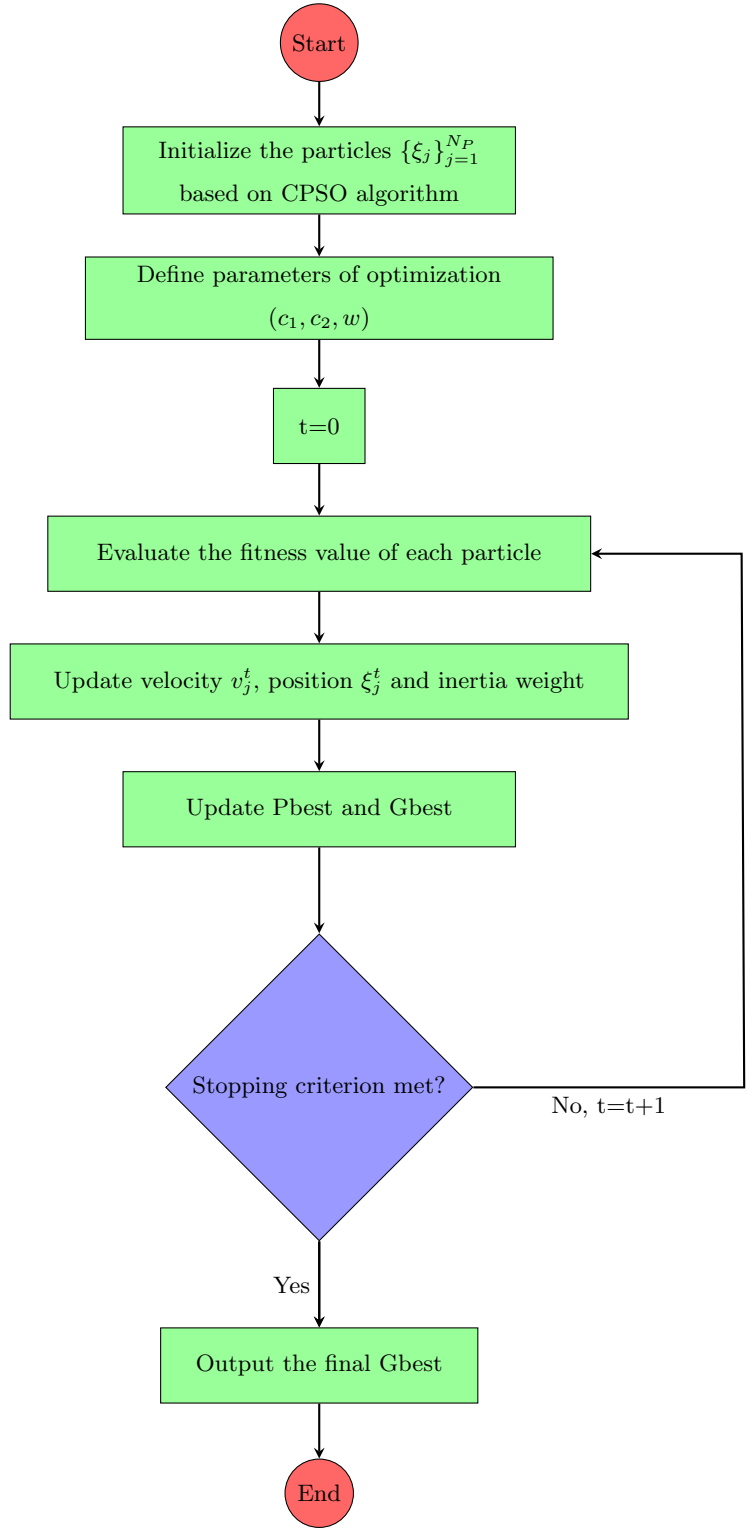
\begin{figure}[!htb]
\begin{center}
\begin{small}

\tikzstyle{startstop} = [circle,draw, draw=black, fill=red!60]

\tikzstyle{io} = [rectangle, 
minimum width=3cm, 
minimum height=1cm, 
text centered, 
text width=5cm, 
draw=black, 
fill=green!40]

\tikzstyle{process} = [rectangle, 
minimum width=3cm, 
minimum height=1cm, 
text centered, 
text width=6cm, 
draw=black, 
fill=green!40]

\tikzstyle{process23} = [rectangle, 
minimum width=1cm, 
minimum height=1cm, 
text centered, 
text width=1cm, 
draw=black, 
fill=green!40]

\tikzstyle{process2} = [rectangle, 
minimum width=3cm, 
minimum height=1cm, 
text centered, 
text width=7cm, 
draw=black, 
fill=green!40]

\tikzstyle{process3} = [rectangle, 
minimum width=3cm, 
minimum height=1cm, 
text centered, 
text width=8cm, 
draw=black, 
fill=green!40]

\tikzstyle{process4} = [rectangle, 
minimum width=3cm, 
minimum height=1cm, 
text centered, 
text width=5cm, 
draw=black, 
fill=green!40]

\tikzstyle{decision} = [diamond, 
minimum width=2cm, 
minimum height=2cm, 
text centered, 
draw=black, 
fill=blue!40]

\tikzstyle{process6} = [rectangle, 
minimum width=3cm, 
minimum height=1cm, 
text centered, 
text width=4cm, 
draw=black, 
fill=green!40]

\tikzstyle{arrow} = [thick,->,>=stealth]

\begin{tikzpicture}[node distance=1.7cm]

\node (start) [startstop] {Start};
\node (in1) [io, below of=start] {Initialize the particles $\{\xi_j\}_{j=1}^{N_P} $ based on  CPSO  algorithm};

\node (pro1) [process, below of=in1] {Define parameters of optimization   \\
  $(c_1, c_2, w)$    };

    \node (pro15) [process23, below of=pro1] {t=0  };
  
  \node (pro11) [process2, below of=pro15] {Evaluate the fitness value of each particle   };

    \node (pro21) [process3, below of=pro11] {Update velocity  $v_j^{t}$, position $\xi_j^{t}$ and inertia weight };

    \node (pro31) [process4, below of=pro21] {Update Pbest and Gbest};
  
\node (dec1) [decision, below of=pro31, yshift=-2cm] {Stopping  criterion  met?};

\node (pro2a) [process6, below of=dec1, yshift=-2cm] {Output the final Gbest };

\node (stop) [startstop, below of=pro2a] {End};

\draw [arrow] (start) -- (in1);

\draw [arrow] (in1) -- (pro1);
\draw [arrow] (pro1) -- (pro15);
\draw [arrow] (pro15) -- (pro11);
\draw [arrow] (pro11) -- (pro21);
\draw [arrow] (pro21) -- (pro31);
\draw [arrow] (pro31) -- (dec1);
\draw [arrow] (dec1) -- (pro2a);
\draw [arrow] (pro2a) -- (stop);
\draw [arrow] (dec1) -- node[anchor=east] {Yes} (pro2a);
\draw [arrow] (dec1) -- node[anchor=north] {No, t=t+1} ([xshift=3.6cm] dec1.east) -- ([xshift=2cm] pro11.east) -- (pro11);
\end{tikzpicture}
\end{small}
\end{center}
\vspace*{-0.3cm} \caption{ \small Flowchart for the PSO algorithm used in this study.}
 \label{a2000}
\end{figure}


\section{Numerical Examples and Performance Evaluation}\label{sec7}

To assess  the  accuracy and validity  of the HRK scheme introduced in this work, we have solved three examples involving  WSIEs.
In this study, we utilized the following combinations of radial kernels:
\begin{align*}
\psi_j^{\varepsilon,\rho}(r) &= e^{-(\varepsilon r)^2}+\rho r^3 
\quad (GA+CU), \\
\psi_j^{\varepsilon,\rho}(r)&= e^{-(\varepsilon r)^2}+\rho r^2 \log(r) \quad (GA+TPS), \\ 
\psi_j^{\varepsilon,\rho}(r) &= \sqrt{ \varepsilon^2  r^2+1}+\rho r^3 \quad (MQ+CU), 
\\ 
\psi_j^{\varepsilon,\rho}(r) &= \sqrt{ \varepsilon^2  r^2+1}+\rho r^2 \log(r)  \quad (MQ+TPS),
\\
\psi_j^{\varepsilon,\rho}(r) &= \frac{1}{\sqrt{ \varepsilon^2  r^2+1}}+\rho r^3 \quad (IMQ+CU), 
\\ 
\psi_j^{\varepsilon,\rho}(r) &= \frac{1}{\sqrt{ \varepsilon^2  r^2+1}}+\rho r^2 \log(r)  \quad (IMQ+TPS). 
\end{align*}
{\color{red} Also, we have optimized the parameters using PSO approach, which is discussed in Section \ref{sec8}.
In computations,  we  set the maximum number of iterations at $T=100$ and the population size at $N_P=20$. }
   Furthermore, the initial parameters are chosen as follows:  $w_{\min}=0.4,$ $w_{\max}=0.9,$ $c_1=1.2 $ and  $c_2=1.7.$ Also, we employ  $10$-point CGL  integration formula with $15$ non-uniform subdivisions (L=15) for the estimation of integrals in the method. 
 {\color{red}   A uniform distribution  in the domain $\Omega$ has been used as evaluation points for calculating the Mean Absolute Error (MAE) and the Root Mean Square Error (RMSE). The MAE measures the average magnitude of absolute differences between the approximated and exact values, whereas the RMSE quantifies the square root of the mean of squared deviations between these two sets of values. For one-dimensional problems, we put $M=200$. For multi-dimensional problems, the number $M$ depends on the irregular domain $\Omega$. In these cases, we firstly established a uniform grid of $40^d$ points, and then  removed  all the points located outside the domain $\Omega$. All routines are written in “Maple 18” software and run on a laptop with an Intel(R) Core(TM) i5-4210U CPU 2.40 GHz
processor and 8GB RAM.
 }

\begin{example} \label{y1}
As the first example let \cite{A17}
\begin{align*}
u(x)-  \int_{0}^{1} \ln \vert x-t \vert u(t) dt=f(x), \quad x \in [0, 1],
\end{align*}
where 
\[  f(x)=( x^3-x ) \Big ( \ln(\frac{1-x}{x}) \Big )+4x^2+\frac{x}{2}-\frac{5}{3},   \]
with the true solution  $u(x)=3x^2-1.$ 
For different $n$ values, the MAE of the suggested hybrid
approach and the standard approach given by \cite{A17} are provided in Table \ref{samplehG}. 
 In \cite{A17}, a median value for $\varepsilon$ was chosen by the authors after a trial-and-error procedure. 
 In addition, Figure \ref{ssin1} illustrates the variations of the condition number for  various choices  of $n$ with different kernels.
Also, Figure \ref{a53} displays the related    RMSE convergence.
In Figures  \ref{ssin1} and \ref{a53}, we utilized $\rho=10^{-8}$ and $\varepsilon=0.2$. 
Figure \ref{ssin1} illustrates that the condition number of the coefficient matrix for the hybrid kernels is lower than that of both the pure MQ and pure GA kernels. Nonetheless, the TPS and CU kernels exhibit the lowest condition number values.
Also, findings in Figure \ref{a53} and Table \ref{samplehG}  demonstrate that  the GA+CU hybrid kernels yield greater accuracy compared to the MQ+TPS.
Furthermore, we find that the hybrid kernels outperform the pure kernels in terms of accuracy.  {\color{red}  From Table \ref{samplehG}, we can see that the elapsed CPU times of the proposed method are lower than the MPI method in \cite{A17}. }
Also, for pure GA kernel and GA+CU hybrid kernel, the root mean square errors and the condition numbers for various choices  of $\rho$  and shape parameter $\varepsilon$ ranging from $0.001$ to $5$ with $n = 10$ on a log-log  coordinate are plotted in Figures \ref{ssin3} and \ref{ssin4},  respectively. 
We see that as the value of $\rho$ increases,  the CU term overcomes  the GA term. As a result, the rate of convergence diminishes, but the stability of the system improves.
 Conversely, as the parameter $\rho$ decreases, the influence of the GA term becomes more significant than that of the CU term, leading to enhanced accuracy of the method while this comes with a rise in the condition number of the system matrix.
Moreover, our observations indicate that the minimum error is typically achieved when the weight parameter $\rho$ is small.
Additionally, Figure \ref{ssin3} illustrates that the pure GA kernel ($\rho=0$) leads to a loss in accuracy for shape parameter $\varepsilon \lesssim 1.2$, while GA+CU hybrid kernel provides  good approximations across all values of the shape parameters. 
In this scheme, the parameter $\rho$ must be
\begin{itemize}
\item[$\bullet$]  adequately small to give faster convergence.
\item[$\bullet$]  sufficiently large to enhance conditioning.
\end{itemize}

 \begin{table}[!htb] \centering  
\caption{\small Comparison of standard method and  hybrid method  in Example    \ref{y1}.} \vspace*{0.1cm} \scalebox{1.1}{ 
\resizebox{\linewidth}{!}{%
\begin{tabular}{lccccccccccccccccc} \cmidrule[1pt](lr){1-15} 
n  &     & \multicolumn{3}{c}{Standard bases in \cite{A17}} &  &&&&&     \multicolumn{3}{c}{Hybrid bases}  \\  
\cmidrule(lr){2-7} \cmidrule(lr){8-15} 
  &   &  GA   &  & &  MQ   &       &  &   &  GA+CU &    &   &  &  MQ+TPS  &    &  &  &   \\  
\cmidrule(lr){2-4} \cmidrule(lr){5-7}  \cmidrule(lr){8-11} \cmidrule(lr){12-15}
  & $\varepsilon=0.15 \times \sqrt{n}$ & $\Vert e_n \Vert_{\infty} $  &{\color{red} Time (s)} & $\varepsilon=\frac{4}{\sqrt{n}}$   &          $\Vert e_n \Vert_{\infty} $ & {\color{red} Time (s)}   &     $\varepsilon_{opt}$ & $\rho_{opt}$   &   $\Vert e_n \Vert_{\infty} $    &{\color{red} Time (s)}   & $\varepsilon_{opt}$  &   $\rho_{opt}$   &  $\Vert e_n \Vert_{\infty} $  & {\color{red} Time (s) } \\ \cmidrule(lr){1-15} 
 4  & {$0.30$}  &   {$4.24 \times 10^{-3}$}  &{\color{red} $0.94 $}    & {$2.00$}  &  {$3.86 \times 10^{-3}$} & {\color{red} $0.98 $}  &   {$0.05 $} & {$5.70 \times 10^{-8}$} & {$4.69 \times 10^{-5}$} &  {\color{red} $0.36 $} & {$1.63$} &  {$6.34 \times 10^{-8}$}   &   {$6.34 \times 10^{-4}$} &  {\color{red} $0.67 $} 
    \\  \cmidrule(lr){1-15} 
 6  & {$0.36$}  &  {$1.32 \times 10^{-4}$} & {\color{red} $4.21 $}  & {$1.63$}  &  {$3.46 \times 10^{-4}$} & {\color{red} $4.32 $} &   {$0.24 $} & {$2.28 \times 10^{-8}$} & {$4.56 \times 10^{-6}$}  &   {\color{red} $2.01 $}   & {$1.12 $}  & {$7.98 \times 10^{-7}$}  &  {$1.89 \times 10^{-5}$} &  {\color{red} $3.78 $} 
    \\   \cmidrule(lr){1-15} 
 8  & {$0.42$}     &  {$6.65 \times 10^{-6}$}  &  {\color{red} $20.67 $} & {$1.41$}  &  {$4.12 \times 10^{-5}$} & {\color{red} $21.03 $} &   {$0.33 $} & {$3.67 \times 10^{-9}$} & {$4.29 \times 10^{-7}$} &    {\color{red} $4.44 $}    &  {$0.91 $} & {$3.67 \times 10^{-8}$}     &  {$7.90 \times 10^{-7}$}  &  {\color{red} $8.29 $} 
    \\    \cmidrule(lr){1-15} 
 10  & {$0.47$}  &  {$4.57 \times 10^{-7}$} & {\color{red} $51.49 $} & {$1.26$}  &  {$7.30 \times 10^{-6}$} & {\color{red} $51.73 $} &   {$0.43 $} & {$3.96 \times 10^{-10}$} & {$3.34 \times 10^{-8}$} &    {\color{red} $7.51 $}    & {$0.79$}  &   {$4.06 \times 10^{-9}$}    &  {$1.20 \times 10^{-7}$} &  {\color{red} $14.45 $} 
    \\  
    \cmidrule[1pt](lr){1-15} 
\end{tabular}}   }
\label{samplehG}
\end{table}

\begin{figure}[!htb] 
\begin{center}
\includegraphics[width=0.9 \textwidth]{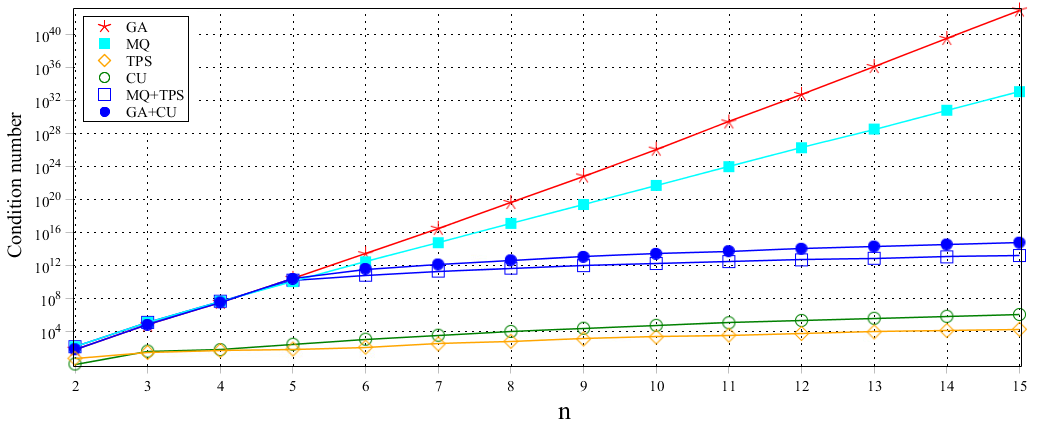}
\end{center}
\vspace*{-0.3cm} \caption{ \small Condition number in Example \ref{y1} for different kernels.}
\label{ssin1}
\end{figure}

\begin{figure}[!htb] 
\begin{center}
\includegraphics[width=0.9 \textwidth]{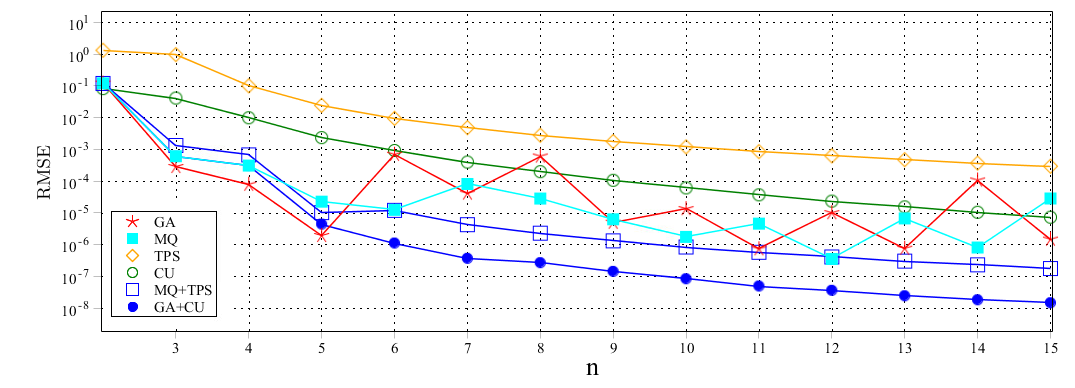}
\end{center}
\vspace*{-0.3cm} \caption{ \small RMSE in Example \ref{y1} for various kernels.}
\label{a53}
\end{figure}

\begin{figure}[!htb] 
\begin{center}
\includegraphics[width=1.05\textwidth]{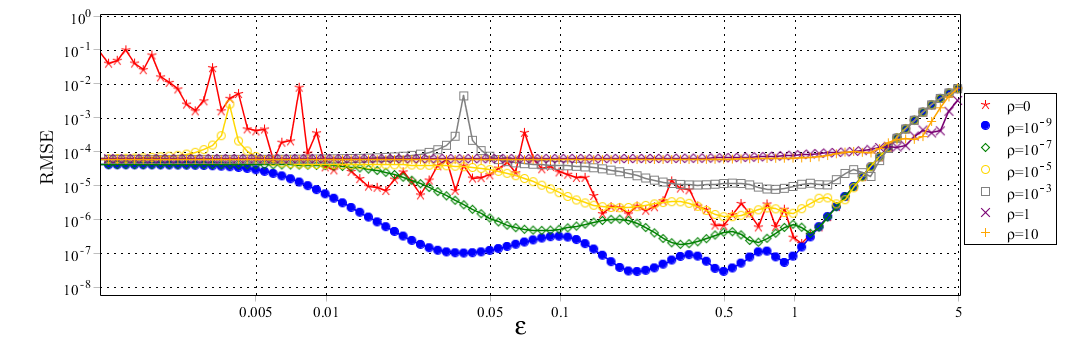}
\end{center}
\vspace*{-0.3cm} \caption{ \small RMSE convergence in Example \ref{y1} for various choices of $\rho$.}
\label{ssin3}
\end{figure}

\begin{figure}[!htb] 
\begin{center}
\includegraphics[width=0.9 \textwidth]{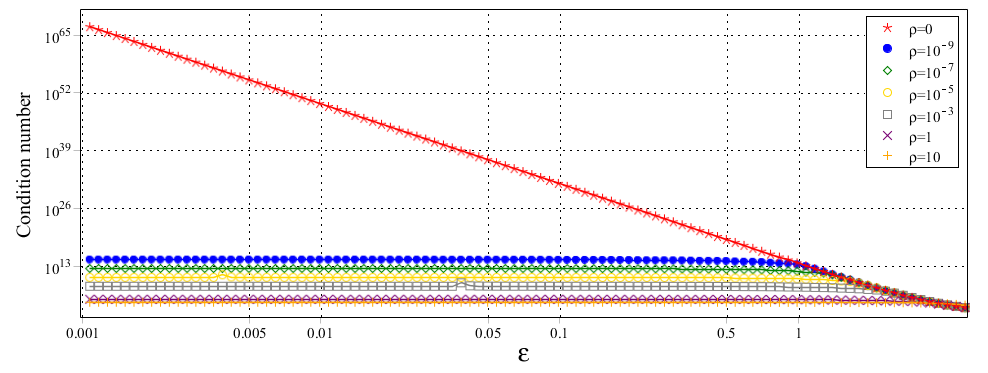}
\end{center}
\vspace*{-0.3cm} \caption{ \small Condition number in Example \ref{y1} for various choices of  $\rho$.}
\label{ssin4}
\end{figure}

\end{example}

\newpage
\vspace{5cm}
\newpage
\vspace{5cm}

\newpage
\vspace{5cm}
\newpage
\hspace*{5cm}
\newpage
\vspace{5cm}
\newpage
\hspace*{5cm}

\begin{example} \label{y9}
We consider the  WSFIE \cite{A17}
\begin{align*}
u(x)-  \int_{0}^{\pi} \ln \vert \cos x-\cos t \vert u(t) dt=1, \quad x \in [0, \pi],
\end{align*}
with true solution $u(x)=\frac{1}{\pi \ln(2)+1} $ \cite{8}.  The kernel $K(x, t) = \ln \vert \cos x-\cos t \vert $ can be rewritten as \cite{8}
\begin{align*}
k(x,t)  
& = \ln \Bigl \vert 2 \sin \frac{1}{2}(x-t) \sin \frac{1}{2}(x+t)  \Bigr \vert \\
& = \ln \Bigl \{ \frac{ 2 \sin \frac{1}{2}(x-t) \sin \frac{1}{2}(x+t)  }{(x-t)(x+t)(2 \pi -x-t)} \Bigr \}  + \ln \vert x-t \vert +\ln (x+t)  + \ln (2 \pi -x-t) \\
& = \sum\limits_{i=1}^{4} R_i (x   ,   t) S_i(x   ,  t), 
\end{align*}
with $R_1 = S_2 = S_3 = S_4 = 1$ and 
\begin{align*}
{} & S_1 (x,t)=\ln \Bigl \{ \frac{2 \sin \frac{1}{2}(x-t) \sin \frac{1}{2}(x+t)  }{(x-t)(x+t)(2 \pi -x-t)} \Bigr \}, \\
 & R_2 (x,t)=\ln \vert x-t \vert,  \\
 &  R_3 (x,t)=\ln (x+t), \\
&  R_4 (x,t)=\ln (2 \pi -x-t). 
\end{align*}
The function $S_1$ is  infinitely differentiable over the interval $[0, 2 \pi ]$, while the functions  $R_2, R_3,$ and $R_4$  are classified as singular functions \cite{8}.  
{\color{red} Here, we firstly converted the interval $[0, \pi]$ to $[0,1]$ via the change of variable $y=\frac{\pi-t}{\pi}$  and then implemented the method described in Subsection \ref{752}.}
The MAE of the proposed hybrid scheme and the standard scheme referenced in \cite{A17} is proposed in Table \ref{samplehG2} for various n values.
 In \cite{A17}, authors considered $\varepsilon=0.15 \times \sqrt{n}$ for GA kernel and $\varepsilon=\frac{4}{\sqrt{n}}$ for IMQ kernel.
Table \ref{samplehG2} illustrates that the error estimates derived from hybrid kernels are  satisfactory and significantly outperform those obtained from pure kernels.
Also, the absolute errors for $n=12$  utilizing pure and hybrid kernels are displayed in Figure   \ref{ssin8}.
 Additionally, Figure \ref{ssin9} illustrates the condition numbers for both the pure and hybrid kernels with  $n=10$. In Figures \ref{ssin8} and \ref{ssin9}, we employed the values $\rho=10^{-10}$ and $\varepsilon=0.2$. 
 The  findings presented here show that using hybrid kernels is appropriate for stable computations in the limit of $\varepsilon \longrightarrow 0.$

 \begin{table}[!htb] \centering  
\caption{\small Comparison of standard method and hybrid  method  in Example    \ref{y9}.} \vspace*{0.1cm} \scalebox{1.08}{ 
\resizebox{\linewidth}{!}{%
\begin{tabular}{lcccccccccccccccccc}   \cmidrule[1pt](lr){1-14}  
n  &     & \multicolumn{3}{l}{Standard bases in \cite{591}} &  &     &&&\multicolumn{3}{c}{Hybrid bases}  \\  
\cmidrule(lr){2-6} \cmidrule(lr){7-14} 
  & GA  &     &    IMQ   &     &  &   &  MQ+CU &    &  &   &  GA+TPS  &    &  &  &   \\  
\cmidrule(lr){2-3} \cmidrule(lr){4-6}  \cmidrule(lr){7-10} \cmidrule(lr){11-14}
  & $\varepsilon=0.1 \times \sqrt{n}$ & $\Vert e_n \Vert_{\infty} $ &  $\varepsilon=1.1+\frac{1}{\sqrt{n}}$   &          $\Vert e_n \Vert_{\infty} $ &  &      $\varepsilon_{opt}$ & $\rho_{opt}$   &   $\Vert e_n \Vert_{\infty} $   &   {\color{red} Time (s)} & $\varepsilon_{opt}$  &   $\rho_{opt}$   &  $\Vert e_n \Vert_{\infty} $ &  {\color{red} Time (s) } \\ \cmidrule(lr){1-14} 
4  & {$0.30$}  &  {$5.92 \times 10^{-4}$}  & {$2.00$}  &  {$4.78 \times 10^{-3}$} &   &   {$1.47 $} & {$9.39 \times 10^{-8}$} & {$4.69 \times 10^{-5}$} &  {\color{red} $0.53 $} & {$0.11$} &  {$2.69 \times 10^{-9}$}   &   {$6.34 \times 10^{-5}$} & {\color{red} $0.73 $}
    \\  \cmidrule(lr){1-14} 
 6  & {$0.36$}  &  {$1.89 \times 10^{-5}$}  & {$1.63$}  &  {$2.22 \times 10^{-4}$} &  &   {$1.28 $} & {$5.81 \times 10^{-9}$} & {$1.19 \times 10^{-7}$} & {\color{red} $2.27 $}   & {$0.19 $}  & {$7.98 \times 10^{-8}$}  &  {$8.59 \times 10^{-7}$}  & {\color{red} $4.41 $}
    \\   \cmidrule(lr){1-14} 
8  & {$0.42$}  &  {$2.89 \times 10^{-7}$}  & {$1.41$}  &  {$1.50 \times 10^{-5}$} &  &   {$0.98$} & {$7.53 \times 10^{-9}$} & {$1.25 \times 10^{-8}$} & {\color{red} $4.64 $} & {$0.25 $} & {$3.67 \times 10^{-8}$}     &  {$7.40 \times 10^{-8}$} & {\color{red} $8.39 $}
    \\    \cmidrule(lr){1-14} 
 10  & {$0.47$}  &  {$7.83 \times 10^{-8}$}  & {$1.26$}  &  {$8.03 \times 10^{-7}$} &  &   {$0.84 $} & {$4.78 \times 10^{-10}$} & {$6.44 \times 10^{-10}$} & {\color{red} $7.18 $} &{$0.36$}  &   {$9.47 \times 10^{-10}$}    &  {$1.23 \times 10^{-9}$} & {\color{red} $15.72 $}
    \\  
  \cmidrule[1pt](lr){1-14} 
\end{tabular}}   }
\label{samplehG2}
\end{table}

 \begin{figure}[!htb] 
\begin{center}
\includegraphics[width=0.8 \textwidth]{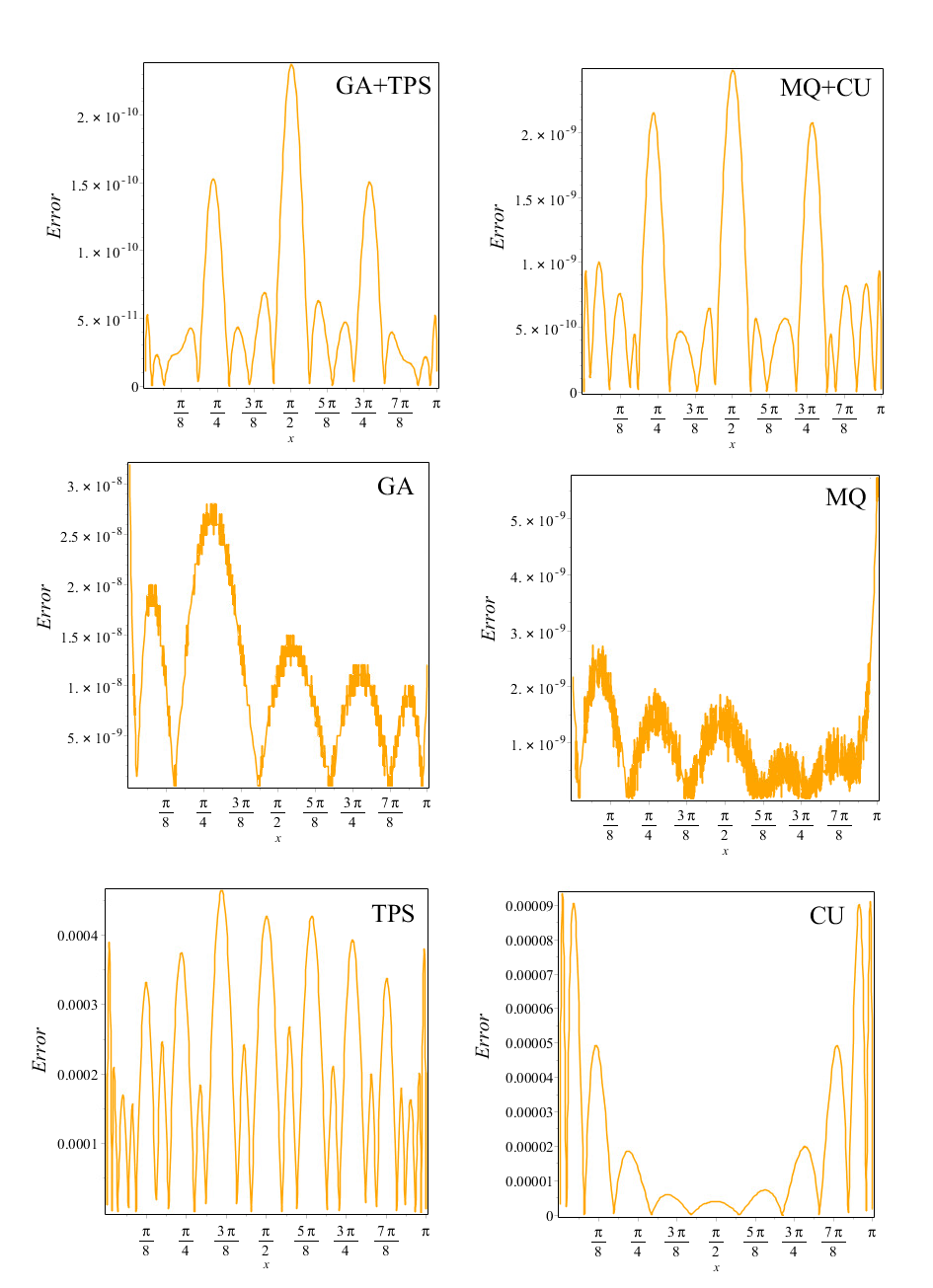}
\end{center}
\vspace*{-0.3cm} \caption{ \small Absolute error of Example \ref{y9} with     $n=12$.}
\label{ssin8}
\end{figure}

 \begin{figure}[!htb] 
\begin{center}
\includegraphics[width=0.9 \textwidth]{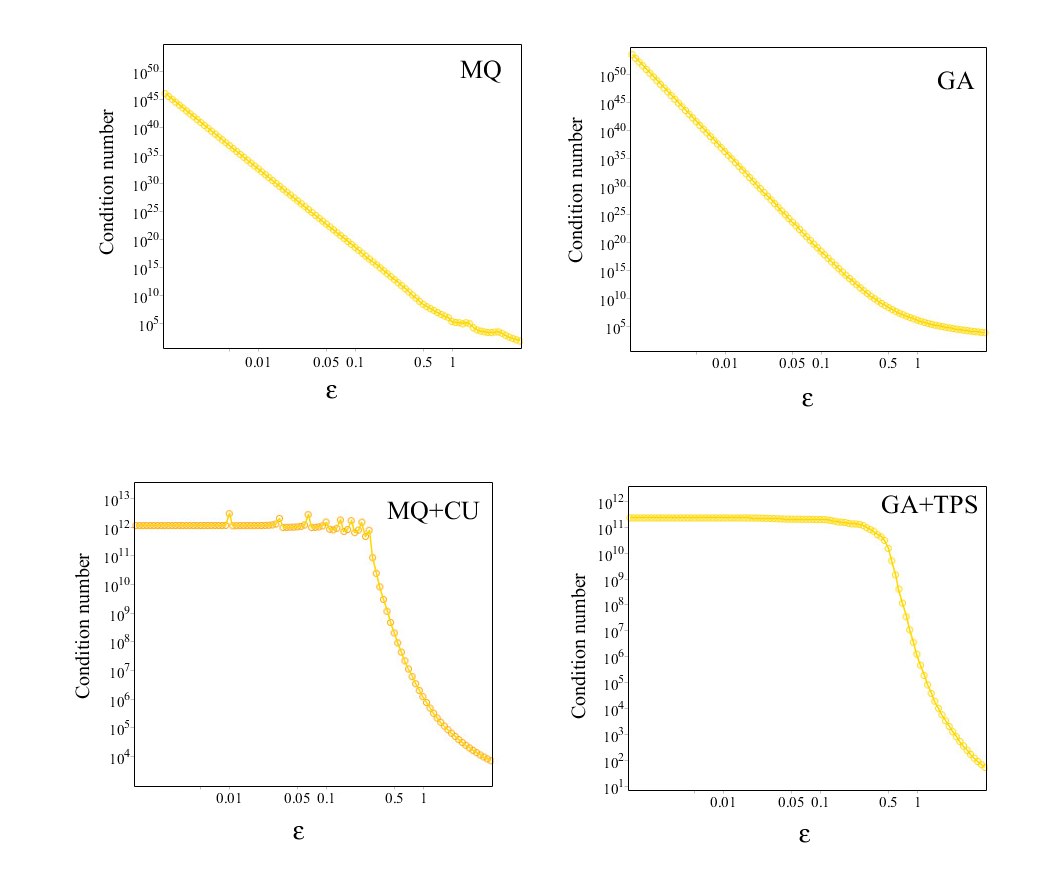}
\end{center}
\vspace*{-0.3cm} \caption{ \small Condition number  in Example \ref{y9} with  $n=10$.}
\label{ssin9}
\end{figure}

\end{example}

\newpage
\vspace{5cm}
\newpage
\vspace{5cm}

\newpage
\vspace{5cm}
\newpage
\hspace*{5cm}
\newpage
\vspace{5cm}
\newpage
\hspace*{5cm}

\begin{example} \label{y10}
We consider the 2D-WSFIE \cite{591}
\begin{align*}
u(x,t)-  \int_{\Omega} \ln \sqrt{(x-y)^2+(t-s)^2} \frac{(x+t)}{\sqrt{s^2+y^2+1}}  u(y,s) dy ds=f(x,t), \quad (x,t) \in  \Omega.
\end{align*}
The function $f$ is designed to find the exact solution $u(x,t)=e^{\frac{x+t-3}{2}}$. Here, $\Omega$ is depicted in Figure \ref{ssin6}  and  is determined by 
\[  \Omega= \Big \{ (y,s) \in \mathbb{R}^2: 0<s<1, 0.5-0.25 \sqrt{3s(3s-3)^2}  < y < 0.5+0.25 \sqrt{3s(3s-3)^2} \Bigr \}.  \]
The traditional approaches for numerically solving this problem encounter challenges because of  the irregular nature of the region. Nevertheless, the approach proposed in this paper, which utilizes a selection of nodes distributed throughout the domain $\Omega$, proves to be effective. 
Figure  \ref{ssin7} shows the distribution of nodes.  {\color{red} Here, a uniform distribution containing $M=1104$ points  in the domain $\Omega$ is used as the evaluation points  for calculating the MAE and RMSE.}
Table \ref{samplehG3} displays the MAE for both the suggested hybrid scheme and the standard  scheme   in \cite{591} across various values of $n$.
 In \cite{591}, authors considered $\varepsilon=0.1 \times \sqrt{n}$ for GA kernel and $\varepsilon=1.1+\frac{1}{\sqrt{n}}$ for MQ kernel. 
Also,  the RMSE convergence for various $n$ values  with different kernels is shown in Figure \ref{FX1}. 
According to Table \ref{samplehG3} and Figure \ref{FX1}, the GA+CU hybrid kernels provide a lower error rate in comparison to the MQ+CU hybrid kernels.
As illustrated in Figure \ref{FX1}, hybrid kernels produce error estimates that are both satisfactory and significantly better than those of pure kernels.
The absolute error with various kernels for  $n=88$ is also graphically displayed in Figure \ref{ssin15}.
Additionally, the condition number variations for various  choices  of n with different kernels are shown in Figure \ref{FX2}.
We find that when hybrid kernels are used, the system matrix's condition number is significantly lower than when  pure MQ and pure GA  kernels are used. Nonetheless, the CU kernels exhibit the lowest condition number values. 
In these {\color{red}figures,} we utilized $\rho=10^{-10}$ and $\varepsilon=0.2$. 
We can generally conclude that the error is significantly reduced when a small portion of a piecewise smooth radial kernel is combined with an infinitely smooth radial kernel.

\begin{figure}[!htb] 
\begin{center}
\includegraphics[width=0.4\textwidth]{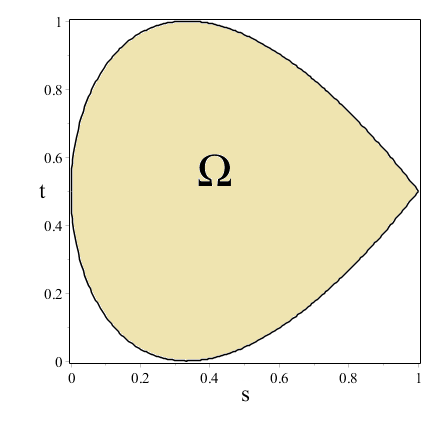}
\end{center}
\vspace*{-0.3cm} \caption{ \small Consideration domain $\Omega$ in Example  \ref{y10}.}
\label{ssin6}
\end{figure}

\begin{figure}[!htb] 
\begin{center}
\includegraphics[width=0.75\textwidth]{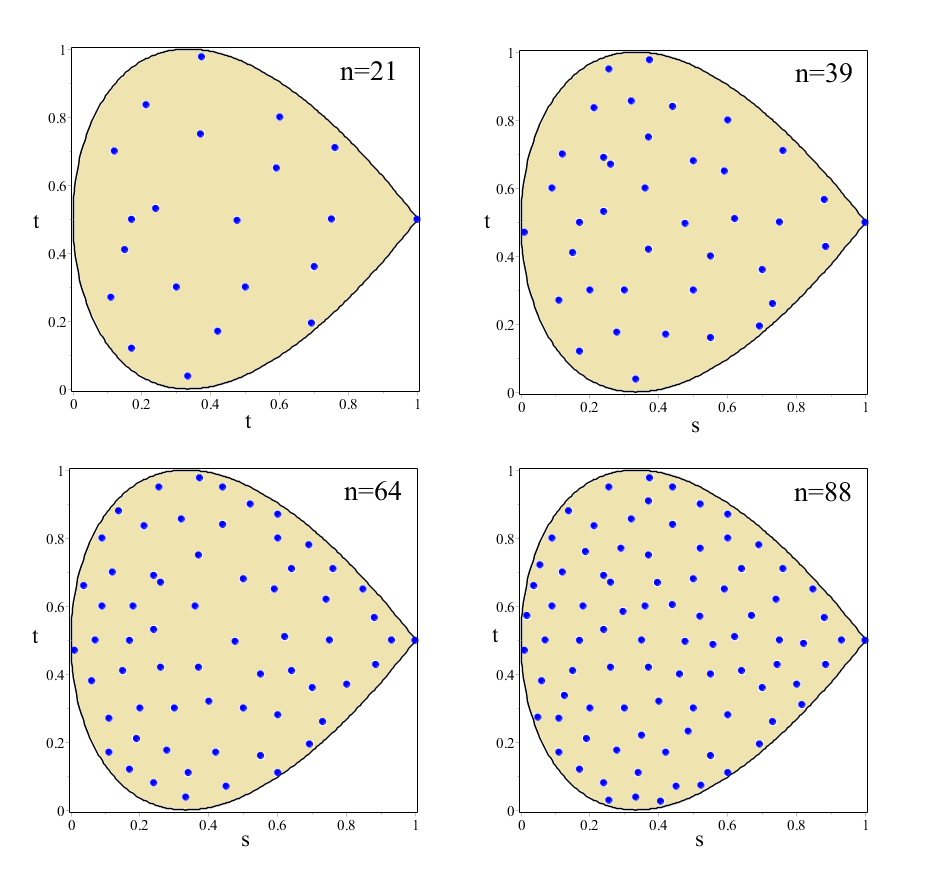}
\end{center}
\vspace*{-0.3cm} \caption{ \small Node distribution in
Example \ref{y10}.}
\label{ssin7}
\end{figure}

 \begin{table}[!htb] \centering  
\caption{\small Comparison of standard method and hybrid  method  in Example    \ref{y10}.} \vspace*{0.1cm} \scalebox{1.08}{ 
\resizebox{\linewidth}{!}{%
\begin{tabular}{lcccccccccccccccccc} \cmidrule[1pt](lr){1-14}   
n  &     & \multicolumn{3}{l}{Standard bases in \cite{591}} &  &     && \multicolumn{3}{c}{Hybrid bases}  \\  
\cmidrule(lr){2-6} \cmidrule(lr){7-14} 
  & GA  &     &    MQ   &     &  &   &  GA+CU &    &  &   &  MQ+CU  &    &  &  &   \\  
\cmidrule(lr){2-3} \cmidrule(lr){4-6}  \cmidrule(lr){7-10} \cmidrule(lr){11-14}
  & $\varepsilon=0.1 \times \sqrt{n}$ & $\Vert e_n \Vert_{\infty} $ &  $\varepsilon=1.1+\frac{1}{\sqrt{n}}$   &          $\Vert e_n \Vert_{\infty} $ &  &      $\varepsilon_{opt}$ & $\rho_{opt}$   &   $\Vert e_n \Vert_{\infty} $   &   {\color{red} Time (s)} & $\varepsilon_{opt}$  &   $\rho_{opt}$   &  $\Vert e_n \Vert_{\infty} $ &  {\color{red} Time (s) } \\ \cmidrule(lr){1-14} 
 11  & {$0.33$}  &  {$3.11 \times 10^{-4}$}  & {$1.40$}  &  {$5.93 \times 10^{-4}$} &   &   {$0.15 $} & {$5.71 \times 10^{-8}$} & {$1.69 \times 10^{-5}$} & {\color{red} 7.41}  & {$1.27$} &  {$2.34 \times 10^{-8}$}   &   {$6.34 \times 10^{-5}$} & {\color{red} 8.21}
    \\  \cmidrule(lr){1-14} 
 21  & {$0.45$}  &  {$2.18 \times 10^{-5}$}  & {$1.31$}  &  {$9.22 \times 10^{-5}$} &  &   {$0.21 $} & {$7.28 \times 10^{-8}$} & {$4.97 \times 10^{-7}$} & {\color{red} 18.14} & {$1.08 $}  & {$1.98 \times 10^{-8}$}  &  {$8.42 \times 10^{-7}$} & {\color{red} 19.74}
    \\   \cmidrule(lr){1-14} 
 39  & {$0.62$}  &  {$3.24 \times 10^{-7}$}  & {$1.26$}  &  {$6.35 \times 10^{-6}$} &  &   {$0.39 $} & {$1.67 \times 10^{-9}$} & {$4.29 \times 10^{-8}$} &  {\color{red} 37.21} & {$0.97 $} & {$4.97 \times 10^{-8}$}     &  {$8.75 \times 10^{-8}$} & {\color{red} 40.21}
    \\    \cmidrule(lr){1-14} 
 64  & {$0.80$}  &  {$1.57 \times 10^{-8}$}  & {$1.22$}  &  {$1.04 \times 10^{-6}$} &  &   {$0.53 $} & {$9.96 \times 10^{-10}$} & {$2.31 \times 10^{-9}$} & {\color{red} 67.89} &{$0.88$}  &   {$3.26 \times 10^{-9}$}    &  {$7.96 \times 10^{-9}$} & {\color{red} 71.27} 
    \\  
      \cmidrule(lr){1-14} 
 88  & {$0.93$}  &  {$2.48 \times 10^{-9}$}  & {$1.20$}  &  {$1.69 \times 10^{-7}$} &  &   {$0.67 $} & {$3.73 \times 10^{-10}$} & {$1.02 \times 10^{-10}$} &  {\color{red} 94.21} & {$0.79$}  &   {$4.07 \times 10^{-10}$}    &  {$5.21 \times 10^{-10}$} & {\color{red} 98.23} 
    \\  
    
 \cmidrule[1pt](lr){1-14}  
\end{tabular}}   }
\label{samplehG3}
\end{table}

\begin{figure}[!htb] 
\begin{center}
\includegraphics[width=0.85\textwidth]{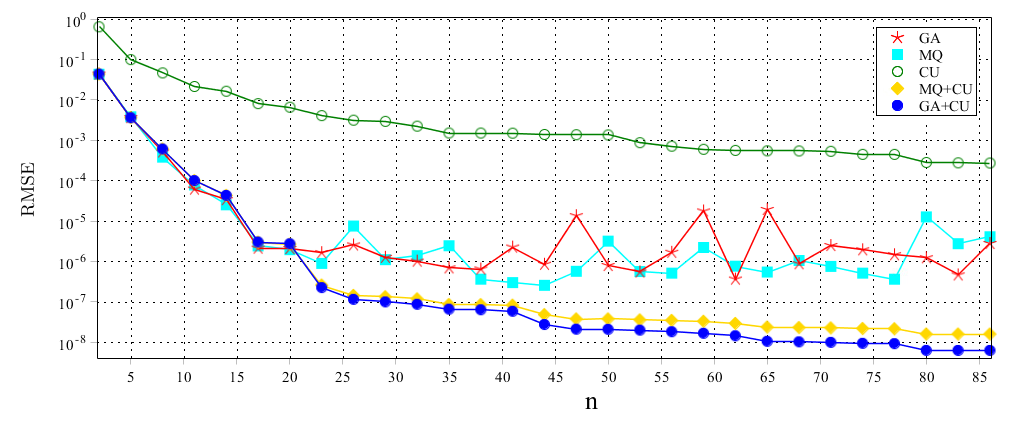}
\end{center}
\vspace*{-0.3cm} \caption{ \small RMSE  in Example \ref{y10} with various kernels.}
\label{FX1}
\end{figure}

\begin{figure}[!htb] 
\begin{center}
\includegraphics[width=0.85\textwidth]{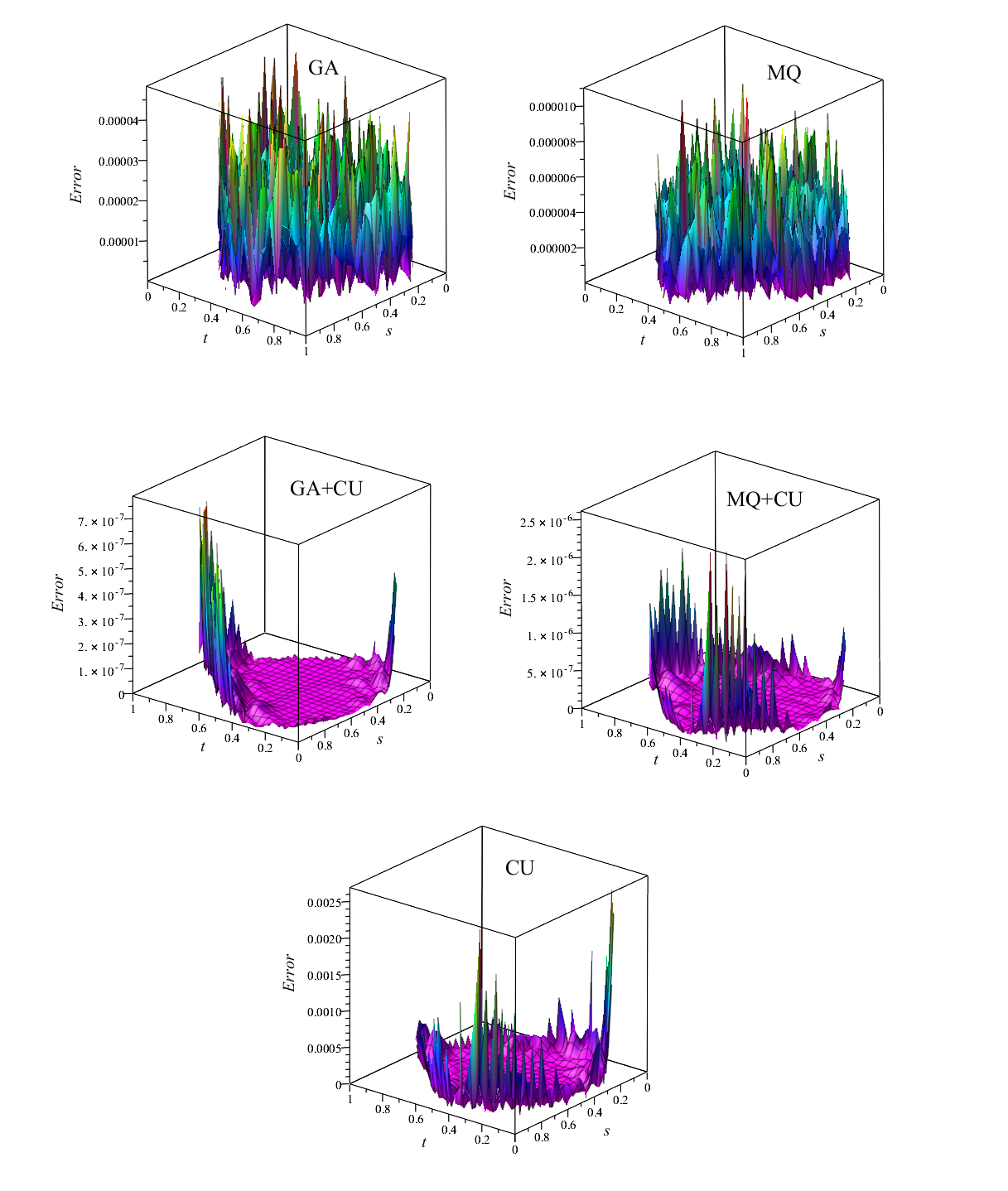}
\end{center}
\vspace*{-0.3cm} \caption{ \small Absolute error of Example \ref{y10}  with $n = 88$.}
\label{ssin15}
\end{figure}

\begin{figure}[!htb] 
\begin{center}
\includegraphics[width=0.85\textwidth]{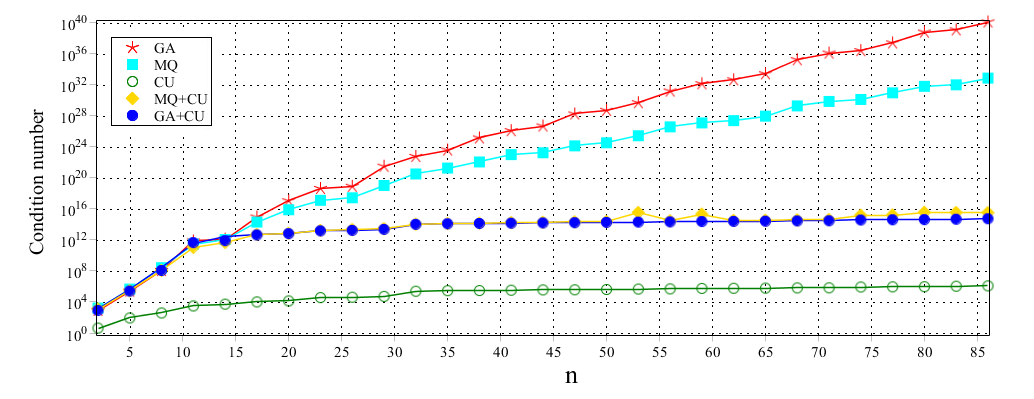}
\end{center}
\vspace*{-0.3cm} \caption{ \small Condition number  in Example \ref{y10} with various kernels. }
\label{FX2}
\end{figure}

\end{example}

\vspace{5cm}

\newpage
\vspace{5cm}

\newpage
\vspace{5cm}
\newpage
\hspace*{5cm}
\newpage
\vspace{5cm}
\newpage
\vspace{5cm}
\newpage
\hspace*{5cm}
{\color{red}
\begin{example} \label{y33}
Consider the following  2D-WSFIE \cite{591}
\begin{align*}
u(x,t)-  \int_{\Omega} \ln \sqrt{(x-y)^2+(t-s)^2} \frac{\sin(x+t)}{(x+t+2)e^{y^2+s^2}}  u(y,s) dy ds=f(x,t), \quad (x,t) \in  \Omega.
\end{align*} 
The function $f$ is designed to find the exact solution $u(x,t)=\frac{sin(x+t+1)}{x+t+1}$.
Here, $\Omega$ is the crescent  domain  which is drawn in Figure \ref{ssin60} and determined by   
\begin{align*} 
D=\Bigl \{(s,t) \in \mathbb{R}^2 : 0 \leqslant t \leqslant 1, \sqrt{0.25-(t-0.5)^2} \leqslant s \leqslant \sqrt{1-4(t-0.5)^2}\Bigl \}.
\end{align*}
The distribution of nodes is depicted in Figure \ref{L64}. {\color{red} Here, a uniform distribution containing $M=1328$ points  in the domain $\Omega$ is used as the evaluation points  for calculating the MAE and RMSE.}
Table \ref{samplehG210} displays the MAE for both the suggested hybrid scheme and the standard  scheme   in \cite{591} across various values of $n$.
Furthermore, the absolute error for different kernels with $n=69$ is graphically shown in Figure \ref{QS1}.
Also, Figure \ref{QS2} displays the condition number variations for different values of $n$ for different kernels. 
We observe that, the condition number of the system matrix using hybrid kernels is much less than the condition number of system matrix obtained using the pure GA and pure IMQ kernels. However, the CU kernel has the lowest values of the condition numbers.
In these figures, we used $\rho=10^{-10}$ and $\varepsilon=0.2$. 
Also, Figure \ref{QS3} and Figure \ref{QS4}, respectively, show the root mean square errors and the condition numbers as functions of $\varepsilon \in [0.01, 2]$ for different values of $\rho$, for pure GA kernel and GA+CU hybrid kernel with $n=69$ on a log-log scale. Figure \ref{QS3} shows that pure GA kernel ($\rho=0$) leads to a loss in accuracy for shape parameter $\varepsilon \lesssim 0.83$, whereas GA+CU hybrid kernel yields useful approximations for all values of shape parameters.
Also, while an decrease in the value of weight parameter from $\rho=1$ to $\rho=10^{-10}$ leads to a decrease of the the minimum root mean square errors from $ \mathcal{O}(10^{-3})$ down to $ \mathcal{O}(10^{-10})$, the further decrease to $\rho=10^{-12}$ does not yield further improvement. It was observed that the error plot suggests smaller values of $\rho$ for better accuracy, whereas the condition number plot suggests the larger values of $\rho$ for better stability.

\begin{figure}[!htb] 
\begin{center}
\includegraphics[width=0.4\textwidth]{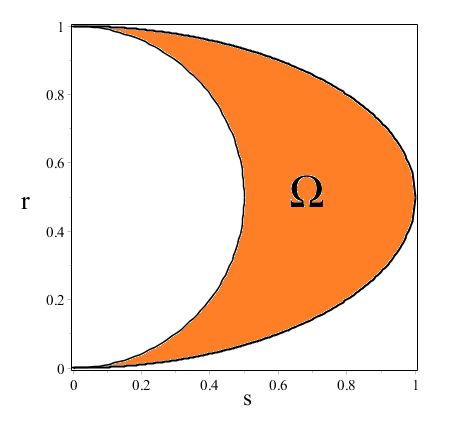}
\end{center}
\vspace*{-0.3cm} \caption{ \small Consideration domain $\Omega$ in Example  \ref{y33}.}
\label{ssin60}
\end{figure}

\begin{figure}[!htb] 
\begin{center}
\includegraphics[width=0.75\textwidth]{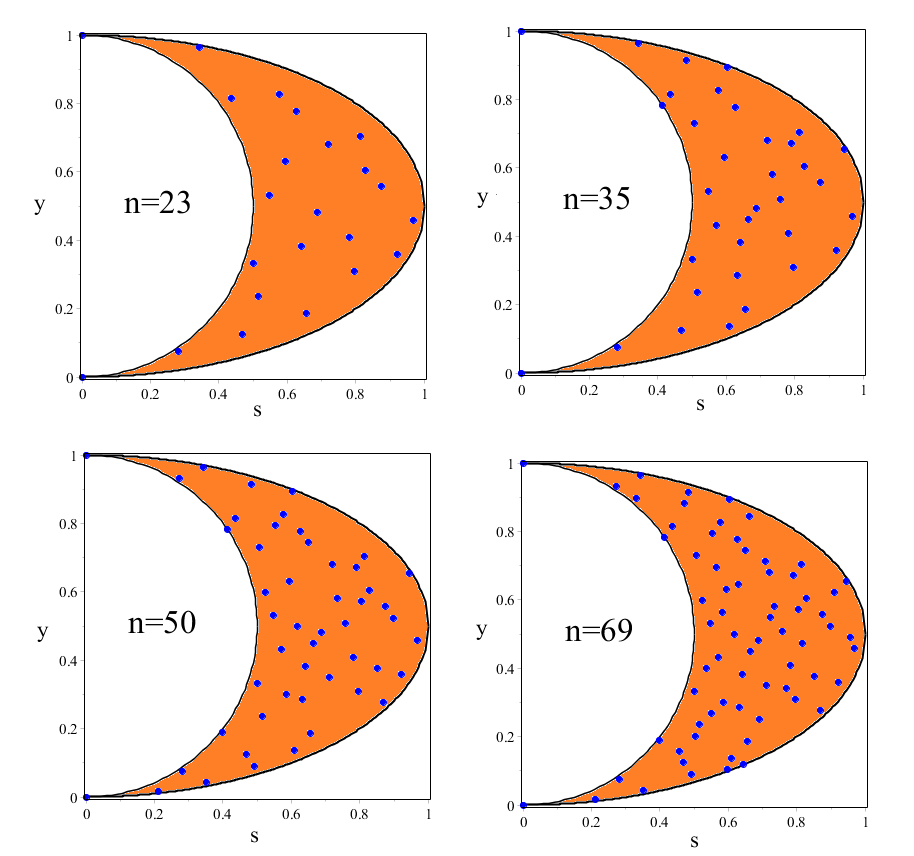}
\end{center}
\vspace*{-0.3cm} \caption{ \small Node distribution in
Example \ref{y33}.}
\label{L64}
\end{figure}

 \begin{table}[!htb] \centering  
\caption{\small Comparison of standard method and hybrid  method  in Example    \ref{y33}.} \vspace*{0.1cm} \scalebox{1.08}{ 
\resizebox{\linewidth}{!}{%
\begin{tabular}{lcccccccccccccccccc}  \cmidrule[1pt](lr){1-14}   
n  &     & \multicolumn{3}{l}{Standard bases in \cite{591}} &  &     &&&\multicolumn{3}{c}{Hybrid bases}  \\  
\cmidrule(lr){2-6} \cmidrule(lr){7-14} 
  & GA  &     &    IMQ   &     &  &   &  GA+CU &    &  &   &  IMQ+CU  &    &  &  &   \\  
\cmidrule(lr){2-3} \cmidrule(lr){4-6}  \cmidrule(lr){7-10} \cmidrule(lr){11-14}
  & $\varepsilon=0.1 \times \sqrt{n}$ & $\Vert e_n \Vert_{\infty} $ &  $\varepsilon=1.1+\frac{1}{\sqrt{n}}$   &          $\Vert e_n \Vert_{\infty} $ &  &      $\varepsilon_{opt}$ & $\rho_{opt}$   &   $\Vert e_n \Vert_{\infty} $   &    Time (s) & $\varepsilon_{opt}$  &   $\rho_{opt}$   &  $\Vert e_n \Vert_{\infty} $ &   Time (s)  \\ \cmidrule(lr){1-14} 
 13  & {$0.36$}  &  {$5.37 \times 10^{-4}$}  & {$1.37$}  &  {$2.56 \times 10^{-3}$} &  &   {$0.13 $} & {$4.37 \times 10^{-8}$} & {$1.07 \times 10^{-4}$} &   $8.67 $ & {$1.29$} &  {$6.39 \times 10^{-8}$}   &   {$4.75 \times 10^{-4}$}&   $9.44 $
    \\  \cmidrule(lr){1-14} 
 23  & {$0.47$}  &  {$9.97 \times 10^{-6}$}  & {$1.30$}  &  {$1.47 \times 10^{-4}$} &  &   {$0.21 $} & {$3.77 \times 10^{-8}$} & {$1.43 \times 10^{-6}$} &  $19.41 $ & {$1.23 $}  & {$2.49 \times 10^{-8}$}  &  {$5.41 \times 10^{-6}$} & $21.84 $
    \\   \cmidrule(lr){1-14} 
 35  & {$0.59$}  &  {$7.65 \times 10^{-7}$}  & {$1.26$}  &  {$3.40 \times 10^{-5}$} &  &   {$0.37 $} & {$1.46 \times 10^{-9}$} & {$3.35 \times 10^{-8}$} &  $33.18 $ & {$1.02 $} & {$5.01 \times 10^{-8}$}     &  {$7.14 \times 10^{-8}$}&  $39.71 $
    \\    \cmidrule(lr){1-14} 
 50  & {$0.70$}  &  {$1.16 \times 10^{-7}$}  & {$1.24$}  &  {$1.22 \times 10^{-5}$} &  &   {$0.51 $} & {$7.86 \times 10^{-9}$} & {$3.79 \times 10^{-9}$} &  $47.08 $ &{$0.89$}  &   {$3.26 \times 10^{-9}$}    &  {$8.43 \times 10^{-9}$} & $54.23 $
    \\  
      \cmidrule(lr){1-14} 
 69  & {$0.83$}  &  {$5.92 \times 10^{-9}$}  & {$1.22$}  &  {$7.98 \times 10^{-7}$} &  &   {$0.66 $} & {$7.19 \times 10^{-10}$} & {$4.18 \times 10^{-10}$} &  $71.35 $ & {$0.69$}  &   {$4.07 \times 10^{-10}$}    &  {$9.64 \times 10^{-10}$} & $77.02$
    \\  
    
     \cmidrule[1pt](lr){1-14} 
\end{tabular}}   }
\label{samplehG210}
\end{table}

\begin{figure}[!htb] 
\begin{center}
\includegraphics[width=0.85\textwidth]{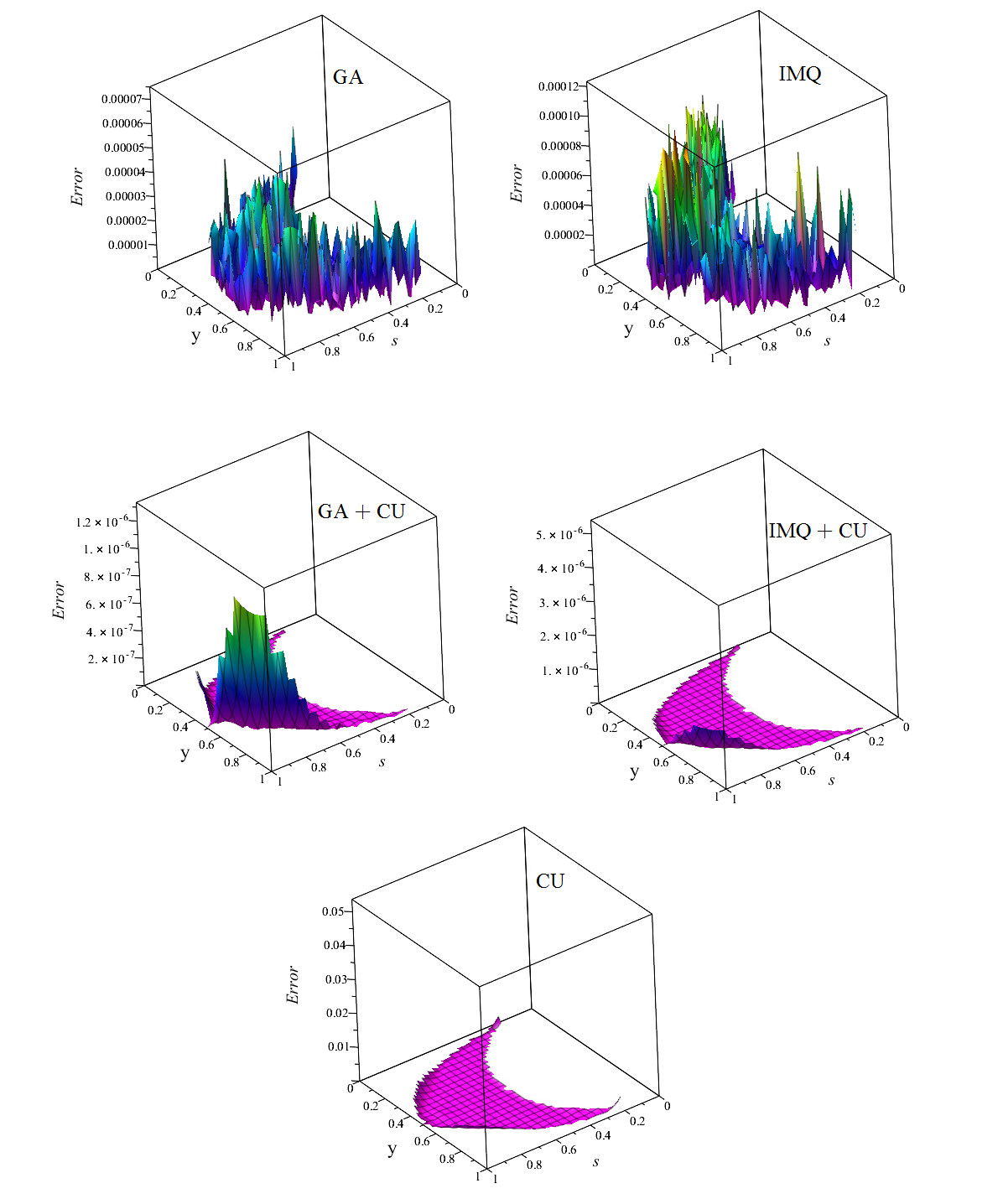}
\end{center}
\vspace*{-0.3cm} \caption{ \small Absolute error of Example \ref{y33} with $n = 69$.}
\label{QS1}
\end{figure}

\begin{figure}[!htb] 
\begin{center}
\includegraphics[width=0.85\textwidth]{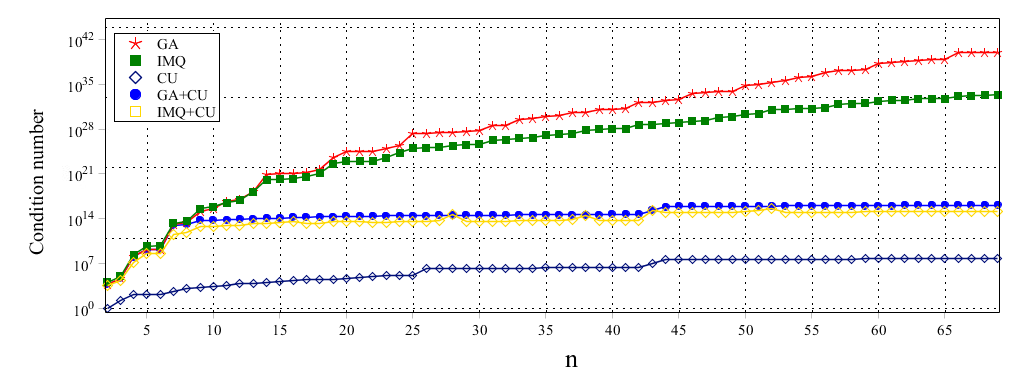}
\end{center}
\vspace*{-0.3cm} \caption{ \small Condition number  in Example \ref{y33} with various kernels. }
\label{QS2}
\end{figure}

\begin{figure}[!htb] 
\begin{center}
\includegraphics[width=0.85\textwidth]{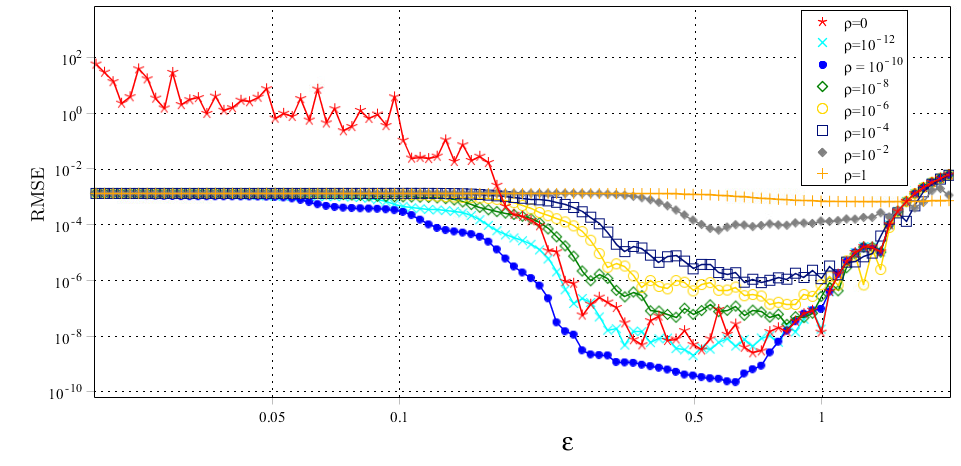}
\end{center}
\vspace*{-0.3cm} \caption{ \small RMSE convergence for different values of  $\rho$. }
\label{QS3}
\end{figure}

\begin{figure}[!htb] 
\begin{center}
\includegraphics[width=0.85\textwidth]{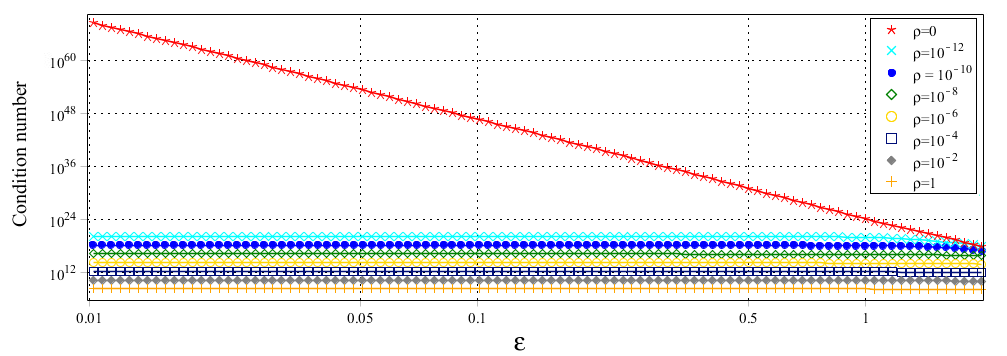}
\end{center}
\vspace*{-0.3cm} \caption{ \small Condition number variations for different values of  $\rho$. }
\label{QS4}
\end{figure}

\end{example}

\newpage
\vspace{5cm}
\newpage

\newpage
\vspace{5cm}
\newpage
\hspace*{5cm}
\vspace{5cm}
\vspace{5cm}
\newpage
\newpage
\vspace{5cm}
\newpage
\newpage
\hspace*{5cm}
\vspace{5cm}
\newpage

\vspace{5cm}
\newpage

\begin{example} \label{y17}
Consider the following  2D-WSFIE \cite{591}
\begin{align*}
u(x,t)-  \int_{\Omega} \ln \sqrt{(x-y)^2+(t-s)^2} \frac{x(1-y^2)}{(1+t)(1+s^2)}  u(y,s) dy ds=f(x,t), \quad (x,t) \in  \Omega.
\end{align*}
The function $f$ is designed to find the exact solution $u(x,t)=\ln (\frac{x^2+1}{t^2+1})$.
Here, $\Omega$ is the annular domain and decomposed by
$\Omega= \Omega_1 \cup  \Omega_2 $ which is drawn in Figure \ref{ssin600} and determined by   
\[  \Omega= \Big \{ (y,s) \in \mathbb{R}^2: (y-0.5)^2+(s-0.5)^2 \leqslant 0.25         \quad \textit{and} \quad  (y-0.5)^2+(s-0.5)^2 \geqslant 0.04        
 \Bigr \}.  \]
 Figure  \ref{ssin700} demonstrates the distribution of collocation points
in the considered domain. {\color{red} Here, a uniform distribution containing $M=1055$ points  in the domain $\Omega$ is used as the evaluation points  for calculating the MAE and RMSE.}
Table \ref{samplehG23} displays the maximum absolute error for both the suggested hybrid scheme and the standard  scheme   in \cite{591} across various values of $n$.
Also, Figure  \ref{ZC1} demonstrates the RMSE convergence for different values of $n$  for various kernels. 
The results presented in Table \ref{samplehG23} and Figure \ref{ZC1} indicate that GA+TPS hybrid bases leads to higher accuracy compared to IMQ+TPS bases.
Furthermore, it has been observed that the hybrid kernels exhibit a higher level of accuracy in comparison to the pure kernels.
The absolute error with various kernels for  $n=89$ is also graphically displayed in Figure \ref{MX10}.
Additionally, Figure \ref{FX20} displays the condition number variations
for different values of  $n$ for different kernels. 
We observe that, the condition number of the system matrix
using hybrid kernels is much less than the condition number of system matrix obtained using the pure GA
and pure IMQ kernels.  Nonetheless, the TPS kernels exhibit the lowest condition number values. 
In these Figures, we utilized $\rho=10^{-9}$ and $\varepsilon=0.25$. 

\begin{figure}[!htb] 
\begin{center}
\includegraphics[width=0.4\textwidth]{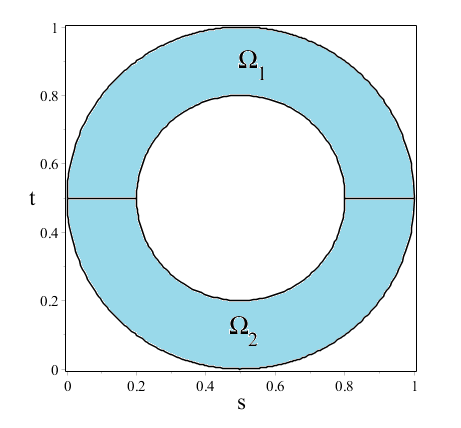}
\end{center}
\vspace*{-0.3cm} \caption{ \small Consideration domain $\Omega$ in Example  \ref{y17}.}
\label{ssin600}
\end{figure}

\begin{figure}[!htb] 
\begin{center}
\includegraphics[width=0.75\textwidth]{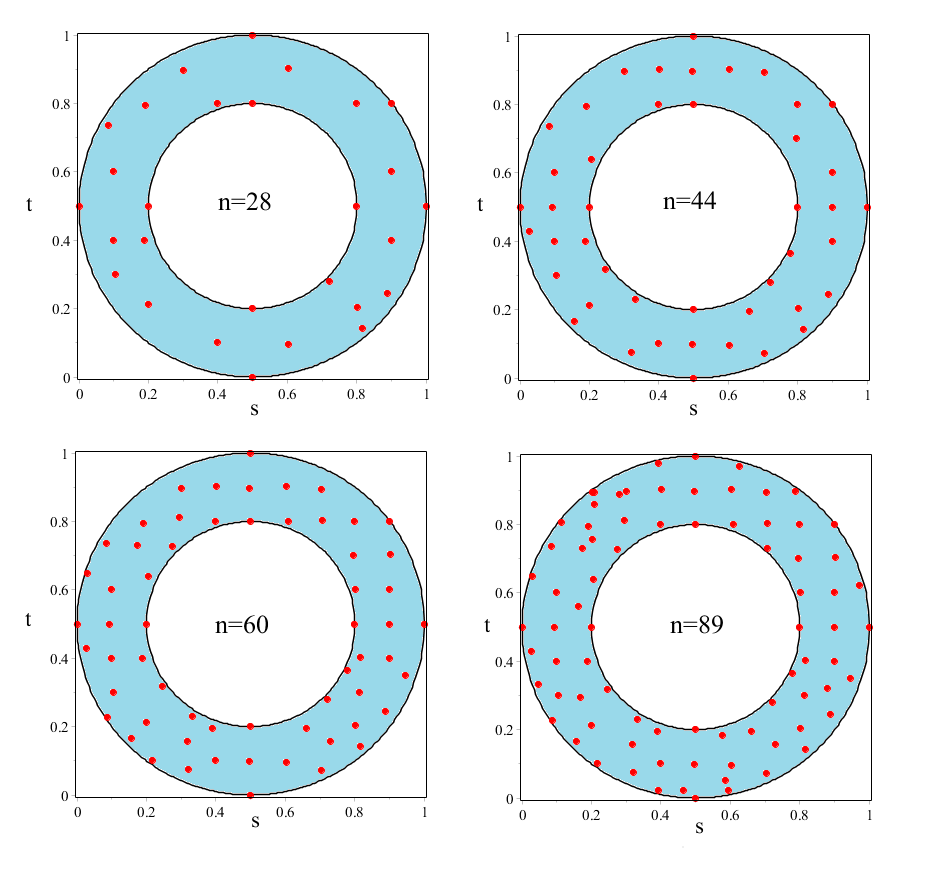}
\end{center}
\vspace*{-0.3cm} \caption{ \small Node distribution in
Example \ref{y17}.}
\label{ssin700}
\end{figure}

 \begin{table}[!htb] \centering  
\caption{\small Comparison of standard method and hybrid  method  in Example    \ref{y17}.} \vspace*{0.1cm} \scalebox{1.08}{ 
\resizebox{\linewidth}{!}{%
\begin{tabular}{lcccccccccccccccccc} \cmidrule[1pt](lr){1-13} 
n  &     & \multicolumn{3}{l}{Standard bases in \cite{591}} &  &     \multicolumn{3}{c}{Hybrid bases}  \\  
\cmidrule(lr){2-6} \cmidrule(lr){7-13} 
  & GA  &     &    IMQ   &     &  &   &  GA+TPS &    &     &  IMQ+TPS  &    &  &  &   \\  
\cmidrule(lr){2-3} \cmidrule(lr){4-6}  \cmidrule(lr){7-9} \cmidrule(lr){10-13}
  & $\varepsilon=0.1 \times \sqrt{n}$ & $\Vert e_n \Vert_{\infty} $ &  $\varepsilon=1.1+\frac{1}{\sqrt{n}}$   &          $\Vert e_n \Vert_{\infty} $ &  &     $\varepsilon_{opt}$ & $\rho_{opt}$   &   $\Vert e_n \Vert_{\infty} $   &    $\varepsilon_{opt}$  &   $\rho_{opt}$   &  $\Vert e_n \Vert_{\infty} $  \\ \cmidrule(lr){1-13} 
 14  & {$0.37$}  &  {$3.28 \times 10^{-3}$}  & {$1.36$}  &  {$3.11 \times 10^{-3}$} &  &   {$0.11 $} & {$1.67 \times 10^{-7}$} & {$5.29 \times 10^{-4}$} &    {$1.27$} &  {$7.83 \times 10^{-8}$}   &   {$6.48 \times 10^{-4}$}
    \\  \cmidrule(lr){1-13} 
 28  & {$0.52$}  &  {$3.51 \times 10^{-4}$}  & {$1.28$}  &  {$1.05 \times 10^{-3}$} &  &   {$0.17 $} & {$2.87 \times 10^{-7}$} & {$3.69 \times 10^{-5}$} &  {$1.04 $}  & {$9.89 \times 10^{-8}$}  &  {$5.21 \times 10^{-5}$} 
    \\   \cmidrule(lr){1-13} 
 44  & {$0.66$}  &  {$1.77 \times 10^{-5}$}  & {$1.25$}  &  {$1.08 \times 10^{-4}$} &  &   {$0.35 $} & {$2.78 \times 10^{-8}$} & {$1.28 \times 10^{-6}$} &  {$0.93 $} & {$5.71 \times 10^{-9}$}     &  {$4.21 \times 10^{-6}$}
    \\    \cmidrule(lr){1-13} 
 60  & {$0.77$}  &  {$2.58 \times 10^{-6}$}  & {$1.22$}  &  {$3.75 \times 10^{-5}$} &  &   {$0.49 $} & {$5.47 \times 10^{-9}$} & {$3.78 \times 10^{-7}$} &  {$0.81$}  &   {$7.14 \times 10^{-9}$}    &  {$5.67 \times 10^{-7}$} 
    \\  
      \cmidrule(lr){1-13} 
 89  & {$0.94$}  &  {$5.01 \times 10^{-8}$}  & {$1.20$}  &  {$2.56 \times 10^{-6}$} &  &   {$0.61 $} & {$1.17 \times 10^{-10}$} & {$7.41 \times 10^{-9}$} &  {$0.75$}  &   {$6.83 \times 10^{-10}$}    &  {$9.21 \times 10^{-9}$} 
    \\  
    
  \cmidrule[1pt](lr){1-13} 
\end{tabular}}   }
\label{samplehG23}
\end{table}

\begin{figure}[!htb] 
\begin{center}
\includegraphics[width=0.85\textwidth]{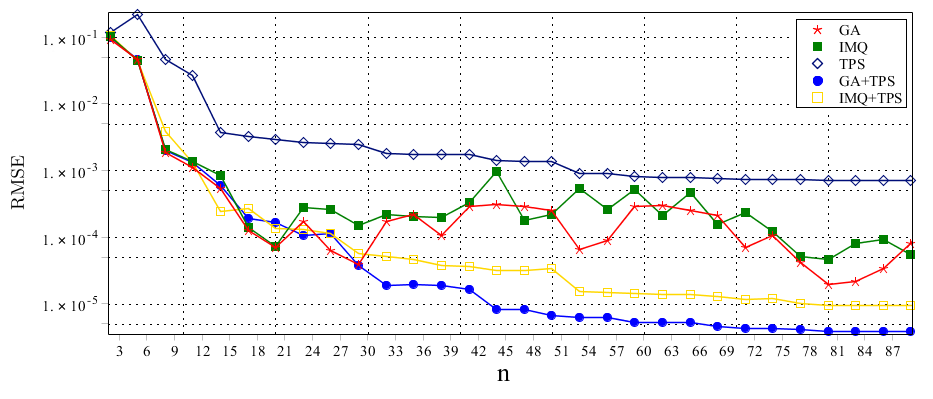}
\end{center}
\vspace*{-0.3cm} \caption{ \small RMSE in Example \ref{y17} with various kernels. }
\label{ZC1}
\end{figure}

\begin{figure}[!htb] 
\begin{center}
\includegraphics[width=0.85\textwidth]{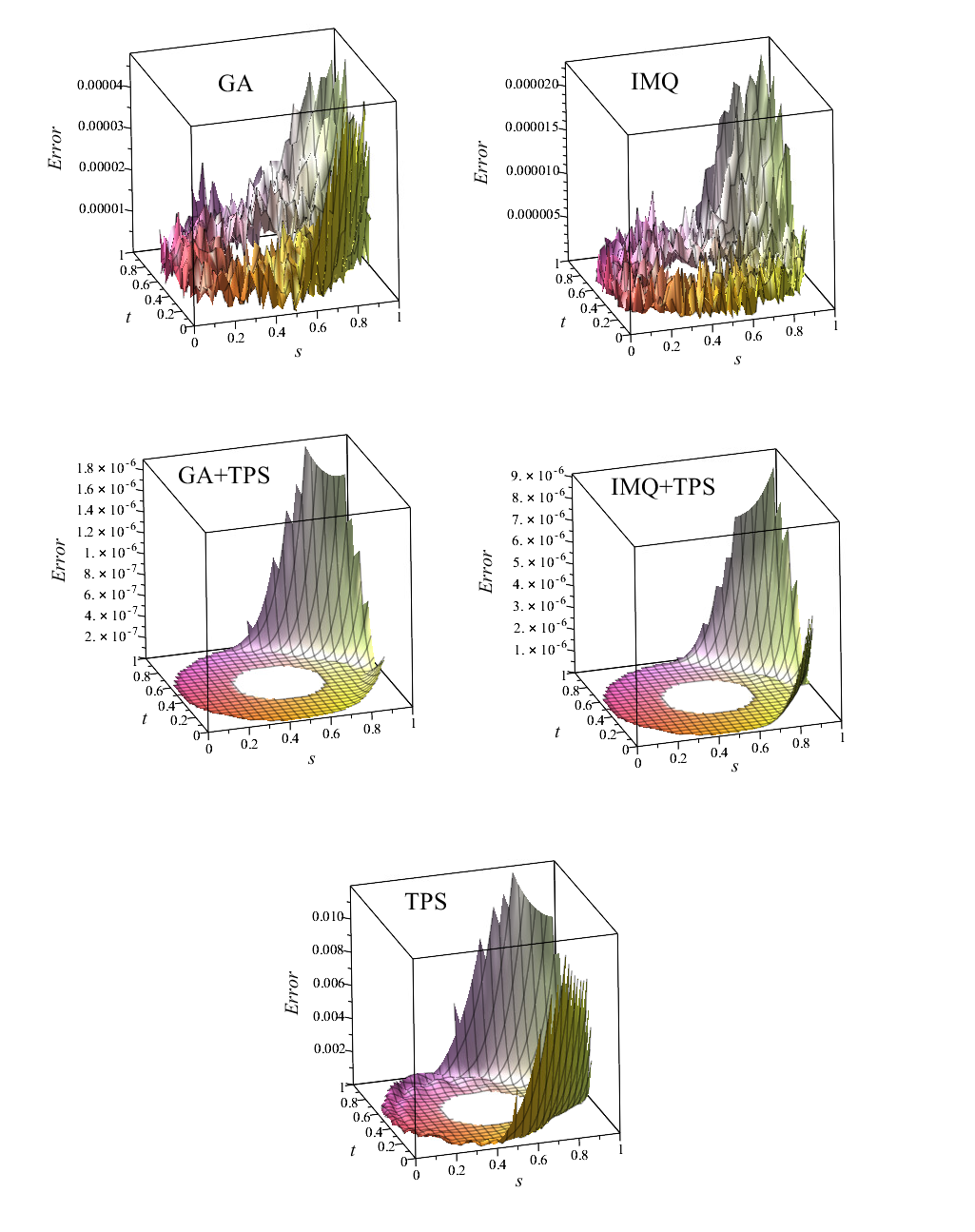}
\end{center}
\vspace*{-0.3cm} \caption{ \small Absolute error of Example \ref{y17} with $n = 89$.}
\label{MX10}
\end{figure}

\begin{figure}[!htb] 
\begin{center}
\includegraphics[width=0.85\textwidth]{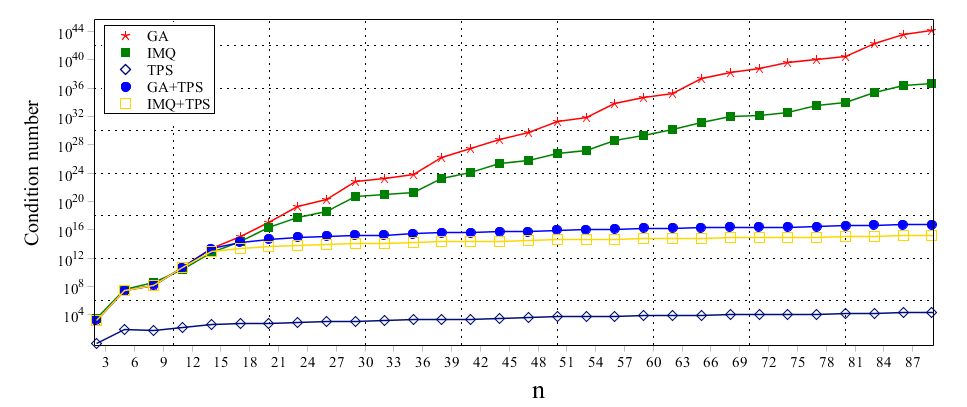}
\end{center}
\vspace*{-0.3cm} \caption{ \small Condition number  in Example \ref{y17} with various kernels. }
\label{FX20}
\end{figure}

\end{example}

\newpage
\vspace{5cm}
\newpage

\newpage
\vspace{5cm}
\newpage
\hspace*{5cm}
\vspace{5cm}

\newpage
\vspace{5cm}
\newpage
\hspace*{5cm}
\newpage
\vspace{5cm}
\newpage
\vspace{5cm}
\newpage

\begin{example} \label{y110}
Consider the following  3D-WSFIE 
\begin{align*}
u(x,t,z)-  \int_{\Omega} \ln \sqrt{(x-y)^2+(t-s)^2+(z-r)^2} \hspace{0.15cm} \frac{y+s+r+1}{e^{x^2+t^2+z^2}}  u(y,s,r) dy ds dr=f(x,t,z), \quad (x,t) \in  \Omega,
\end{align*}
where the function $f$ has been so chosen that the exact
solution of this equation is $u(x,t,z)=\frac{1}{x+t+z+1}$.
The domain $\Omega=\displaystyle \bigcup_{i=1}^{4} \Omega_i$ is shown in Figure \ref{FT1} and is defined as follows:
\begin{align*}
\Omega_1 &=\Bigl \{(r,s,y) : 0 \leqslant r \leqslant \frac{3}{10}, 0 \leqslant s \leqslant \frac{9}{10}, 0 \leqslant y \leqslant \frac{3}{10} \Bigl \}, \vspace*{0.15cm} \\
\Omega_2 &=\Bigl \{(r,s,y) : \frac{3}{10} \leqslant r \leqslant \frac{6}{10}, -r+ \frac{9}{10} \leqslant s \leqslant -r+ \frac{12}{10}, 0 \leqslant y \leqslant \frac{3}{10} \Bigl \}, \vspace*{0.15cm} \\
\Omega_3 &=\Bigl \{(r,s,y) : \frac{6}{10} \leqslant r \leqslant \frac{9}{10}, r- \frac{3}{10} \leqslant s \leqslant r, 0 \leqslant y \leqslant \frac{3}{10} \Bigl \}, \vspace*{0.15cm} \\
\Omega_4 &=\Bigl \{(r,s,y) : \frac{9}{10} \leqslant r \leqslant \frac{12}{10}, 0 \leqslant s \leqslant \frac{9}{10}, 0 \leqslant y \leqslant \frac{3}{10} \Bigl \}.
\end{align*}
 Figure  \ref{FT2} demonstrates the distribution of collocation points
in the considered domain. {\color{red} Here, a uniform distribution containing $M=13824$ points  in the domain $\Omega$ is used as the evaluation points  for calculating the MAE and RMSE.}
 Result of PSO approach  for various $n$ values  are
documented in Table  \ref{samplehGA}.
Also, Figure  \ref{AN1} demonstrates the RMSE convergence for different values of $n$  for various kernels. 
According to Table \ref{samplehGA} and Figure \ref{AN1}, the GA+CU hybrid kernels provide a lower error rate in comparison to the IMQ+CU hybrid kernels.
As illustrated in Figure \ref{AN1}, hybrid kernels produce error estimates that are both satisfactory and significantly better than those of pure kernels.
Additionally, Figure \ref{AN2} displays the condition number variations
for different values of  $n$ for different kernels. 
It can be seen that  the condition number of the system matrix
using hybrid kernels is much less than the condition number of system matrix obtained using the pure GA
and pure IMQ kernels.  Nonetheless, the CU kernels exhibit the lowest condition number values. 
In these Figures, we utilized $\rho=10^{-10}$ and $\varepsilon=0.25$. 
Overall, proposed method with simple hybridized kernel gives a reasonable comprise both on accuracy and conditioning.
\begin{figure}[!htb] 
\begin{center}
\includegraphics[width=0.4\textwidth]{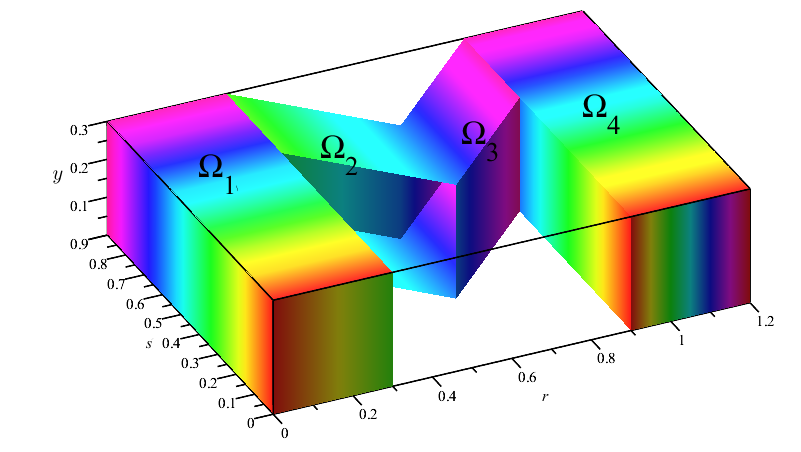}
\end{center}
\vspace*{-0.3cm} \caption{ \small Consideration domain $\Omega$ in Example  \ref{y110}.}
\label{FT1}
\end{figure}

\begin{figure}[!htb] 
\begin{center}
\includegraphics[width=0.75\textwidth]{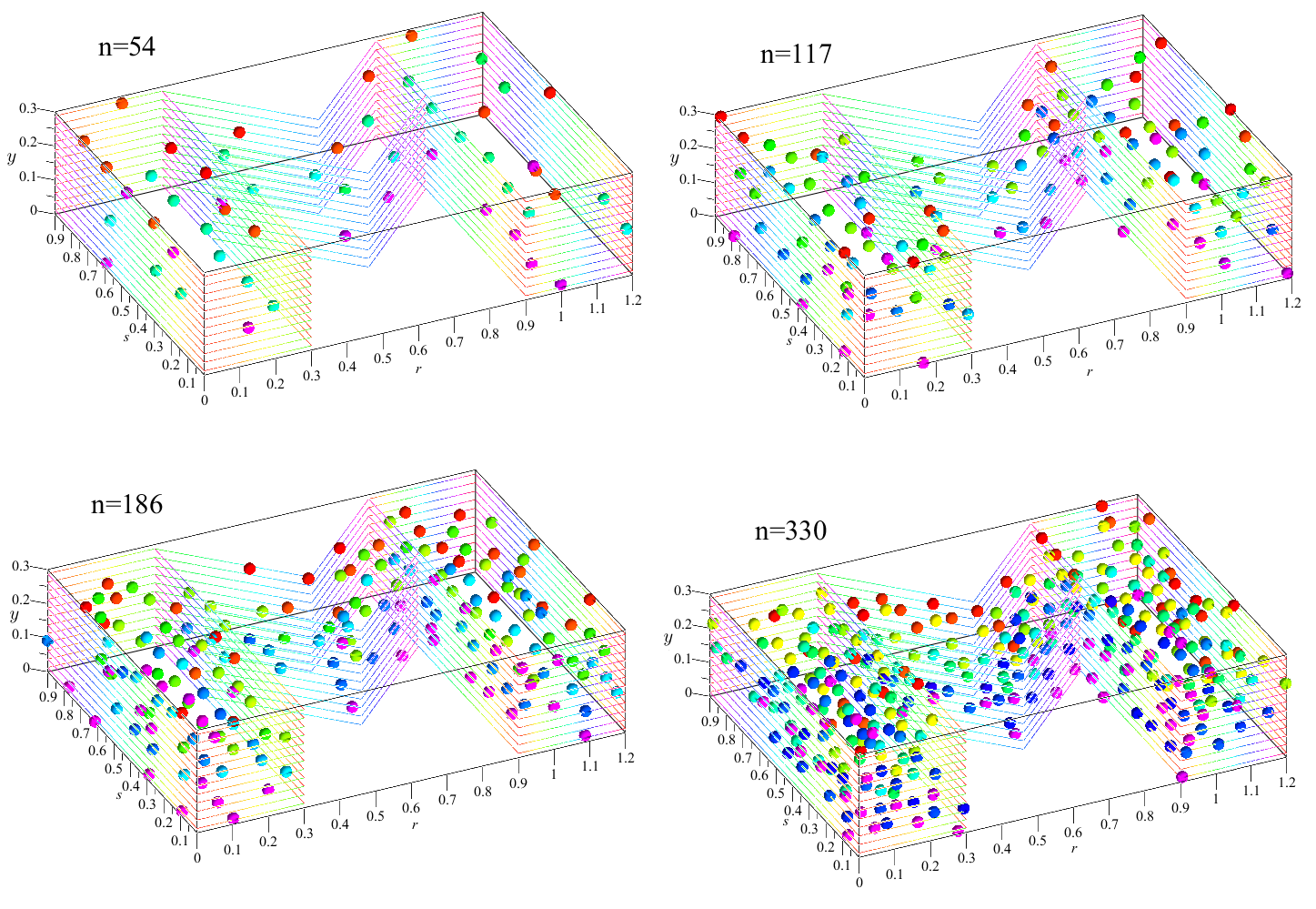}
\end{center}
\vspace*{-0.3cm} \caption{ \small Node distribution in
Example \ref{y110}.}
\label{FT2}
\end{figure}

 \begin{table}[!htb] \centering  
\caption{\small Comparison of standard method and hybrid  method  in Example    \ref{y110}.} \vspace*{0.1cm} \scalebox{1.08}{ 
\resizebox{\linewidth}{!}{%
\begin{tabular}{lcccccccccccccccccc} \cmidrule[1pt](lr){1-13}  
n  &     &  \multicolumn{3}{l}{ \quad  Standard bases} &  &     \multicolumn{3}{c}{\quad  \quad  \quad \quad   \hspace*{2cm} Hybrid bases}  \\  
\cmidrule(lr){2-6} \cmidrule(lr){7-13} 
  & GA  &     &    IMQ   &     &  &   &  GA+CU &    &     &  IMQ+CU  &    &  &  &   \\  
\cmidrule(lr){2-3} \cmidrule(lr){4-6}  \cmidrule(lr){7-9} \cmidrule(lr){10-13}
  & $\varepsilon_{opt}$ & $\Vert e_n \Vert_{\infty} $ &  $\varepsilon_{opt}$   &          $\Vert e_n \Vert_{\infty} $ &  &     $\varepsilon_{opt}$ & $\rho_{opt}$   &   $\Vert e_n \Vert_{\infty} $   &    $\varepsilon_{opt}$  &   $\rho_{opt}$   &  $\Vert e_n \Vert_{\infty} $  \\ \cmidrule(lr){1-13} 
 39  & {$0.42$}  &  {$1.98 \times 10^{-3}$}  & {$1.74$}  &  {$4.12 \times 10^{-2}$} &  &   {$0.07 $} & {$4.77 \times 10^{-8}$} & {$3.27 \times 10^{-4}$} &    {$1.27$} &  {$1.18 \times 10^{-8}$}   &   {$6.24 \times 10^{-4}$}
    \\  \cmidrule(lr){1-13} 
 54  & {$0.62$}  &  {$2.83 \times 10^{-4}$}  & {$1.63$}  &  {$2.07 \times 10^{-3}$} &  &   {$0.14 $} & {$6.34 \times 10^{-8}$} & {$2.61 \times 10^{-5}$} &  {$1.04 $}  & {$4.67 \times 10^{-8}$}  &  {$6.42 \times 10^{-5}$} 
    \\   \cmidrule(lr){1-13} 
 117  & {$0.76$}  &  {$3.37 \times 10^{-5}$}  & {$1.48$}  &  {$1.76 \times 10^{-4}$} &  &   {$0.26 $} & {$5.37 \times 10^{-9}$} & {$2.38 \times 10^{-6}$} &  {$0.93 $} & {$8.97 \times 10^{-9}$}     &  {$7.11 \times 10^{-6}$}
    \\    \cmidrule(lr){1-13} 
 186  & {$0.93$}  &  {$5.61 \times 10^{-6}$}  & {$1.37$}  &  {$3.47 \times 10^{-5}$} &  &   {$0.38 $} & {$4.81 \times 10^{-9}$} & {$3.29 \times 10^{-7}$} &  {$0.81$}  &   {$2.97 \times 10^{-10}$}    &  {$8.63 \times 10^{-7}$} 
    \\  
      \cmidrule(lr){1-13} 
 330  & {$1.12$}  &  {$7.23 \times 10^{-7}$}  & {$1.24$}  &  {$2.86 \times 10^{-6}$} &  &   {$0.57 $} & {$2.24 \times 10^{-10}$} & {$2.48 \times 10^{-8}$} &  {$0.75$}  &   {$4.48 \times 10^{-10}$}    &  {$9.21 \times 10^{-8}$} 
    \\  
    
  \cmidrule[1pt](lr){1-13} 
\end{tabular}}   }
\label{samplehGA}
\end{table}

\begin{figure}[!htb] 
\begin{center}
\includegraphics[width=0.95\textwidth]{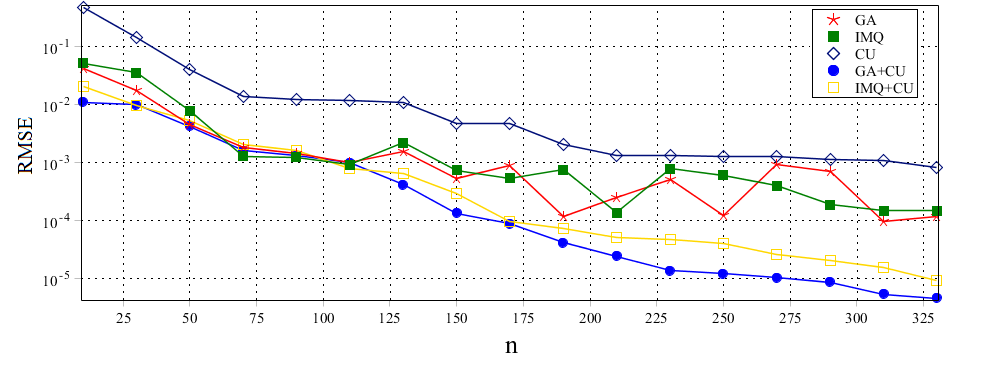}
\end{center}
\vspace*{-0.3cm} \caption{ \small RMSE  in Example \ref{y110} with various kernels.}
\label{AN1}
\end{figure}

\begin{figure}[!htb] 
\begin{center}
\includegraphics[width=0.95\textwidth]{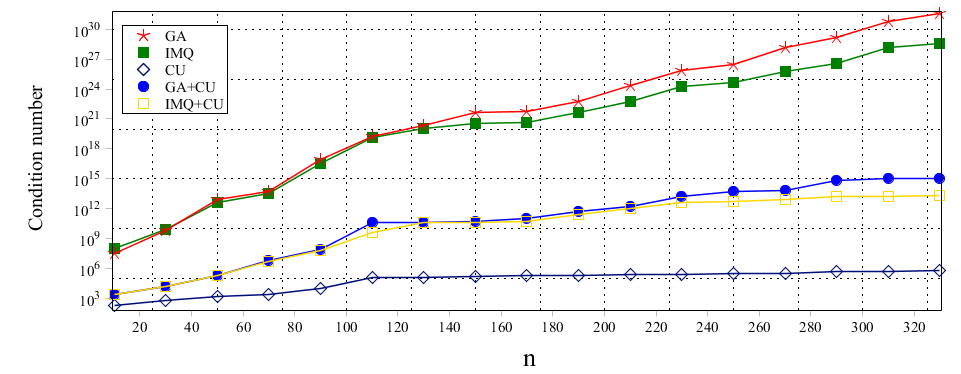}
\end{center}
\vspace*{-0.3cm} \caption{ \small Condition number  in Example \ref{y110} with various kernels. }
\label{AN2}
\end{figure}

\end{example}

}

\vspace{5cm}

\newpage
\vspace{5cm}
\newpage
\hspace*{5cm}
\newpage
\vspace{5cm}

\newpage

\section{Concluding remarks and future directions }\label{sec9}
 In this work, an efficient interpolation approach based on HRKs has been  provided to find the approximate solution of the linear WSFIEs of the second kind.  The technique employs hybrid kernels built on  {\color{red} scattered } nodes as a basis in the collocation method, where all integrals involved in this method are approximated using the non-uniform  CGL integration rule.
Hybrid kernels enhance the algorithm's stability and flexibility, allowing it to handle computations efficiently even with very small shape parameters as well as with relatively larger degrees of freedom.
By applying this method, solving mentioned  IEs are  converted to the solving a system of linear equations.
Also, {\color{red}a } convergence analysis of the presented method was proved. 
 Moreover, we employed the PSO algorithm to determine the  shape parameter and weight coefficient  in hybrid kernels when numerically approximating the WSFIEs.  Furthermore, we conducted a comparison between the numerical results achieved using pure and hybrid  bases. The numerical simulations indicate that as $n$ increases, the hybrid kernels exhibit greater accuracy compared to the pure kernels. Also,  numerical findings demonstrate that a tiny doping of the piecewise smooth radial kernels into the infinitely smooth radial kernels diminishes the condition number of the system matrix, effectively addressing the ill-conditioned issue associated with the infinitely smooth radial kernels. Also, the presented hybrid algorithm do not need a structured mesh, and therefore be applied to approximate complex geometry problems using a set of dispersed nodes.
In conclusion, the  proposed algorithm is a valuable tool that opens up new avenue for solving different problems in the field of science and engineering.

\section*{Declaration of Competing Interest}
The authors declare that they have no known competing financial interests or personal relationships that could
have appeared to influence the work reported in this paper.
\section*{Acknowledgments}
The authors sincerely thank the anonymous referees for their valuable comments and suggestions, which significantly improved the quality of this paper.

\end{document}